\documentclass[10pt, a4wide, a4paper,reqno]{amsart}
\usepackage{amsmath, amsthm, amssymb, latexsym, mathrsfs, amsfonts, cases,  indentfirst, 
}
\usepackage{color}

\usepackage{amsmath, amssymb}

\usepackage{amsfonts}

\usepackage{color}      
\usepackage{indentfirst}        
\usepackage{latexsym, bm}        

\usepackage{graphicx}
\usepackage{cases}
\usepackage{fancyhdr}
\usepackage{pifont}
\newtheorem{theorem}{Theorem}
\newtheorem{defi}[theorem]{Definition}

\newtheorem{proposition}[theorem]{Proposition}

\newtheorem{remark}[theorem]{Remark}
\usepackage{enumerate}
\numberwithin{equation}{section}

\usepackage{tikz-cd}
\allowdisplaybreaks[4]
\begin{document}
	\title[$RLL$-Realization and its Hopf superalgebra structure of
	$U_{p, q}(\widehat{\mathfrak{gl}}(m|n))$]{$RLL$-Realization and its Hopf superalgebra structure of
		$U_{p, q}(\widehat{\mathfrak{gl}}(m|n))$}	
	\author[N.H. Hu]{Naihong Hu}
	\address{School of Mathematical Sciences, MOE Key Laboratory of Mathematics and Engineering Applications \& Shanghai Key Laboratory of PMMP, East China Normal University, Shanghai 200241, China}
	\email{nhhu@math.ecnu.edu.cn}
	
	\author[N. Jing]{Naihuan Jing}
	\address{Department of Mathematics, North Carolina State University, Raleigh, NC 27695, USA}
	\email{jing@ncsu.edu}
	
	\author[Xin Zhong]{Xin Zhong$^*$}
	\address{School of Mathematical Sciences, MOE Key Laboratory of Mathematics and Engineering Applications \& Shanghai Key Laboratory of PMMP, East China Normal University, Shanghai 200241, China}
	\email{52215500002@stu.ecnu.edu.cn}
	
	\thanks{$^*$Corresponding author.\\
		The paper is supported by the NNSFC (Grant Nos.
		12171155, 12171303) and in part by the Science and Technology Commission of Shanghai Municipality (No. 22DZ2229014)
		and Simons Foundation TSM-00002518}
	\begin{abstract}
		In this paper, we extend the Reshetikhin-Semenov-Tian-Shansky formulation of quantum affine algebras to
the two-parameter quantum affine superalgebra $U_{p, q}(\widehat{\mathfrak{gl}}(m|n))$
and obtain its Drinfeld realization. We also derive its Hopf algebra structure by providing
 Drinfeld-type coproduct for the Drinfeld generators.
	\end{abstract}
	
	\keywords{Basic $R$-matrix; Drinfeld realization; $RLL$-formulation; quantum affine super algebra; Gauss decomposition; Hopf algebra structure}
	\subjclass[2010]{Primary 17B37; Secondary  16T05}
	\maketitle
	
	
	\section{Introduction}
	The quantum affine algebras $U_q(\widehat{\mathfrak{g}})$ are the quantum enveloping algebras of the affine Kac-Moody algebras $\widehat{\mathfrak{g}}$ introduced by Drinfeld and Jimbo using the Chevalley generators and $q$-Serre relations. There are two other presentations of $U_q(\widehat{\mathfrak{g}})$: the Faddeev-Reshetikhin-Takhtajan formalism or the $R$-matrix realization given by Reshetikhin and Semenov-Tian-Shansky \cite{RS} using the spectral-parameter-dependent Yang-Baxter equation in the framework of the quantum inverse scattering method \cite{FRT, TF}. The third realization is Drinfeld's new realization, a $q$-analogue of the current algebra realization of the classical affine Lie algebra \cite{D}.
	
	The isomorphism between the Drinfeld realization and the Drinfeld-Jimbo presentation for the quantum affine algebra and the Yangian
	was announced by Drinfeld \cite{D} and a detailed construction of the isomorphism for $U_q(\widehat{\mathfrak g})$ was given by Beck \cite{B}.
	It was shown by  Ding and Frenkel that they are also equivalent to the $R$-matrix formalism
	for the quantum affine algebra in type $A$ by using the Gauss decomposition, which also provides a natural way to derive the Drinfeld realization from the $R$-matrix realization of the quantum affine algebra \cite{DF}. Similarly Brundan and Kleshchev proved the isomorphism for the Yangian algebra in type $A$ \cite{BK}. Jing, Liu, and Molev extended this result to types $B_n^{(1)} , C_n^{(1)}, D_n^{(1)}$ both for the Yangians and quantum affine algebras \cite{JLM1, JLM2, JLM3}. We remark that a correspondence between the $R$-matrix formalism and Drinfeld realization was given for the quantum affine algebras of untwisted types in \cite{HM}, and recently the isomorphism was also established for the twisted quantum affine algebra in type $A_{2 n-1}^{(2)}$ \cite{JZL}.
	
	In 2008, Hu, Rosso and Zhang originally defined the two-parameter quantum affine algebra $U_{r, s}(\widehat{\mathfrak{s l}}_n)$ and derived its Drinfeld realization, where they also introduced the quantum affine Lyndon basis \cite{HRZ}. It was the presentation of the basic $R$-matrix of $U_{r,s}(\mathfrak {gl}_n)$ given by Benkart-Witherspoon \cite{BW} that Jing and Liu were able to recover the Drinfeld realization of two-parameter quantum affine algebra $U_{r, s}(\widehat{\mathfrak{g l}}_n)$ via the $RLL$ formalism \cite{JL}. For $B_n^{(1)}$, $C_n^{(1)}, D_n^{(1)}$ types, Hu-Xu-Zhuang, Zhong-Hu-Jing and Zhuang-Hu-Xu recently worked out two-parameter basic $R$-matrices (including finite and affine case) via a representation-theoretic way in the respective cases and obtained the two-parameter versions of their $RLL$-realizations \cite{ZHJ, ZHX, HXZ}.
	
	In \cite{FHS, Zh}, the authors extended the Reshetikhin and Semenov-Tian-Shansky ($RS$) construction to supersymmetric cases. Based on this and the super version of the Ding-Frenkel theorem, they obtained the defining relations for $U_q(\widehat{\mathfrak{gl}}(m|n))$ in terms of super (or graded) current generators (generating functions). In \cite{WLZ}, Wu-Lin-Zhang  extended the investigations of Jing, Liu, and Molev concerning quantum affine algebras in types $(BCD)_n^{(1)}$ to the supersymmetric case and obtained the $R$-matrix presentation of $U_q(\mathfrak{osp}(2 m+1|2 n))$.
	However, few studies have been done for  the two-parameter quantum affine superalgebra $U_{p, q}(\widehat{\mathfrak{gl}}(m|n))$. In particular,
one can ask whether the two-parameter quantum affine superalgebra $U_{p, q}(\widehat{\mathfrak{gl}}(m|n))$
	can be formulated in the Reshetikhin and Semenov-Tian-Shansky construction.

	In this paper, we will obtain the Drinfeld realization of the two-parameter quantum affine superalgebra $U_{p, q}(\widehat{\mathfrak{gl}}(m|n))$ by using the  $RS$ superalgebra and a super analogue of the Gauss decomposition formula. Compared with the standard quantum affine superalgebra $U_q(\widehat{\mathfrak{gl}}(m|n))$, the two-parameter case 
has twice as many group-like generators, 
it requires more effort to control additional terms in the quantum $R$-matrix. As a result there are more
relations from the $RLL$ formalism,  most notably are the Serre relations (\ref{ser rel1 N})-(\ref{ser rel8 N}) and the commutation relations between $X_i^{+}(z)$ (\ref{com rels1 Xi})-(\ref{com rels7 Xi}) with some coefficients involving the factor $qp^{-1}$ (whereas in one-parameter case this degenerates to $1$).

One of the motivations for Drinfeld to introduce quantum groups was finding new examples of noncommutative and noncommutative Hopf algebras.
  A major advantage of the R-matrix formulism of quantum groups is its natural and systematic construction of the Hopf algebra structures, which are often
  noncommutative and noncommutative. In \cite{Zh}, Zhang presented a simple coproduct for the quantum current generators for the Hopf algebra of $U_q(\widehat{\mathfrak{gl}}(m|n))$. Our method extends this construction to the two-parameter situation, and the RLL construction also provides
  a natural Hopf algebra structure for the quantum affine superalgebra $U_{q, p}(\widehat{\mathfrak{gl}}(m|n))$, which is a generalization of Drinfeld's
  coproduct \cite{D}. As in the one-parameter case, this coproduct is also defined over the
  completion of $U_{q, p}(\widehat{\mathfrak{gl}}(m|n))^{\otimes 2}$.
Based on the explicit coproduct formulae of all the Drinfeld generators, we will construct certain Hopf skew-pairing and get the universal $R$-matrix for $U_{p, q}(\widehat{\mathfrak{gl}}(m|n))$ in a forthcoming paper. This will enable us to further establish an explicit isomorphism between the $R$-matrix and Drinfeld presentations of the quantum affine superalgebra in type $A$.
	
	This paper is organized as follows. In section 2, we start with the basic definitions and notations about the $RS$ superalgebra and listing all super $RLL$ formalism. Section 3 is devoted to quantum affine $R$-matrices and RLL relations for $U_{p, q}(\widehat{\mathfrak{gl}}(m|n))$. We first study the case of $N=2$  in details (including $m=0, n=2$, $m=1, n=1$ and $m=0, n=2$), the second step is to treat the case of $N=3$. This leads to the general case of $N$ to study the commutation relations between the Gaussian generators and to deduce the Drinfeld realization of $U_{p, q}(\widehat{\mathfrak{gl}}(m|n))$. In section 4, the coproduct, counit and antipode for the current realization of $U_{p, q}(\widehat{\mathfrak{gl}}(m|n))$ are given, thus establishing a Hopf superalgebra structure of this quantum affine superalgebra.
	
	\section{Preliminaries}
	\subsection{ $RS$ superalgebra and quantum
		affine superalgebras}
	For a given quantum affine superalgebra, let us define the Reshetikhin-Semenov-Tian-Shansky ($RS$) superalgebras.
Let $V$ be a superspace graded by $\mathbb Z_2$, and consider the tensor space $V\otimes V$. The tensor product follows the usual super rule: for homogeneous elements $a, b, a^{\prime}$ $b^{\prime}$
	$$
	(a \otimes b)\left(a^{\prime} \otimes b^{\prime}\right)=(-1)^{[b][a']}\left(a a^{\prime} \otimes b b^{\prime}\right),
	$$
	where $[a] \in \mathbb{Z}_2$ denotes the degree of $a$. Let $P$ be the permutation operator on
the tensor product $V \otimes V$ given by
$$P\left(v_\alpha \otimes v_\beta\right)=(-1)^{[\alpha][\beta]}\left(v_\beta \otimes v_\alpha\right), \qquad \forall\, v_\alpha, v_\beta \in V.
$$

Let $R(z) \in \operatorname{End}(V \otimes V)$ 
be a matrix obeying the weight conservation condition that $R(z)_{\alpha \beta, \alpha^{\prime} \beta^{\prime}}=0$ unless $\left[\alpha^{\prime}\right]+\left[\beta^{\prime}\right]+[\alpha]+[\beta]=0 \bmod 2$, and satisfying the graded Yang-Baxter equation (YBE) in $End(V^{\otimes 3})$:
	\begin{align}\label{e:YBE}
		R_{12}(z) R_{13}(z w) R_{23}(w)=R_{23}(w) R_{13}(z w) R_{12}(z),
	\end{align}
where $R_{12}=R\otimes I$, $R_{23}=1\otimes R$, and $R_{13}=P_{23}(R_{23})$. The R-matrix also satisfies the unitary condition
	\begin{equation}\label{e:R2}
		R_{12}(z) R_{21}\left(z^{-1}\right)=1.
	\end{equation}
and the symmetry relations:	
	\begin{align}\label{e:R3}
		&P_{12} R_{12}(z) P_{12}=R_{21}(z),\\ \label{e:R4}
		&R_{12}\left(\frac{z}{w}\right) R_{21}\left(\frac{w}{z}\right)=1.
	\end{align}
	\begin{defi}\cite{FHS}
		Let $R\left(\frac{z}{w}\right)$ be a $R$-matrix satisfying the graded YBE.  The RS  superalgebra $U(\mathcal{R})$ is generated by the elements $\ell_{ij}^{\pm}(z)$ satisfying the following relations in terms of the matrices $L^{ \pm}(z)=(\ell_{ij}^{\pm}(z))$:
		\begin{align}\label{e:LL1}
			R\left(\frac{z}{w}\right) L_1^{ \pm}(z) L_2^{ \pm}(w) & =L_2^{ \pm}(w) L_1^{ \pm}(z) R\left(\frac{z}{w}\right),\\ \label{e:LL2}
			R\left(\frac{z_{+}}{w_{-}}\right) L_1^{+}(z) L_2^{-}(w) & =L_2^{-}(w) L_1^{+}(z) R\left(\frac{z_{-}}{w_{+}}\right),
		\end{align}
		where $L_1^{ \pm}(z)=L^{ \pm}(z) \otimes 1, L_2^{ \pm}(z)=1 \otimes L^{ \pm}(z)$, and $z_{ \pm}=z q^{ \pm \frac{c}{2}}$.
The expansion direction of $R\left(\frac{z}{w}\right)$ can be chosen in $\frac{z}{w}$ or $\frac{w}{z}$ in \eqref{e:LL1}
and only in $\frac{z}{w}$ in \eqref{e:LL2}.
	\end{defi}
	The superalgebra $U(\mathcal{R})$  is a graded Hopf algebra, its coproduct and antipode are defined by
	\begin{align}\label{e:Hopf1}
		\Delta\left(L^{ \pm}(z)\right)&=L^{ \pm}\left(z q^{ \pm 1 \otimes \frac{c}{2}}\right) \dot{\otimes} L^{ \pm}\left(z q^{\mp \frac{c}{2} \otimes 1}\right),\\
		S\left(L^{ \pm}(z)\right)&=L^{ \pm}(z)^{-1}.
	\end{align}
where the first formula means that $\Delta(l^{\pm}_{ij}(z))=\sum_k l^{\pm}_{ik}(zq^{ \pm 1 \otimes \frac{c}{2}})\otimes l^{\pm}_{kj}(zq^{\mp \frac{c}{2} \otimes 1})$.
	The Ding-Frenkel decomposition theorem is still true in our case, so we can obtain the defining relations of quantum affine superalgebra $U_{p, q}(\widehat{\mathfrak{gl}}(m|n))$.
	
	\begin{proposition} 
\cite{FHS} The generating matrices
		$L^{ \pm}(z)$ can be written uniquely as follows. 
		\begin{equation}
			L^{ \pm}(z)=\left(\begin{array}{cccc}
				1 & \cdots & & 0 \\
				e_{2,1}^{ \pm}(z) & \ddots & & \\
				e_{3,1}^{ \pm}(z) & & & \\
				\vdots & & & \\
				e_{N, 1}^{ \pm}(z) & \cdots & e_{N, N-1}^{ \pm}(z) & 1
			\end{array}\right)\left(\begin{array}{ccc}
				k_1^{ \pm}(z) & \cdots & 0 \\
				\vdots & \ddots & \vdots \\
				0 & \cdots & k_N^{ \pm}(z)
			\end{array}\right)
		\end{equation}
		$$
		\times\left(\begin{array}{ccccc}
			1 & f_{1,2}^{ \pm}(z) & f_{1,3}^{ \pm}(z) & \cdots & f_{1, M}^{ \pm}(z) \\
			\vdots & \ddots & \cdots & & \vdots \\
			& & & & f_{M-1, M}^{ \pm}(z) \\
			0 & & & & 1
		\end{array}\right).
		$$
		where $e_{i, j}^{ \pm}(z), f_{j, i}^{ \pm}(z)$ and $k_i^{ \pm}(z)$ $(i>j)$  are elements in RS superalgebra $U(\mathcal R)$ and $k_i^{ \pm}(z)$ are invertible. Define
		\begin{gather}\label{e:Xpm1}
			X_i^ - (z) = f_{i,i + 1}^ + \left( {{z_ + }} \right) - f_{i,i + 1}^ - \left( {{z_ - }} \right),
			\\ \label{e:Xpm2}
			X_i^ + (z) = e_{i + 1,i}^ + \left( {{z_ - }} \right) - e_{i + 1,i}^ - \left( {{z_ + }} \right).
		\end{gather}
		Then $q^{ \pm c / 2}, X_i^{ \pm}(z), k_j^{ \pm}(z), \  i=1, \ldots, N-1, j=1, \ldots, N$ generate
$U_{p, q}(\widehat{\mathfrak{gl}}(m|n))$.
	\end{proposition}
	
	To find other relations, it is convenient to write the inverses of $L^{ \pm}(z)$ using the Gauss decomposition:
	\begin{align}\notag
		L^{ \pm}(z)^{-1}= & \left(\begin{array}{cccc}
			1 & -f_1^{ \pm}(z) & \cdots & \\
			\vdots & \ddots & & \vdots \\
			& & & -f_{N-1}^{ \pm}(z) \\
			0 & \cdots & & 1
		\end{array}\right)\left(\begin{array}{ccc}
			k_1^{ \pm}(z)^{-1} & \cdots & 0 \\
			\vdots & \ddots & \vdots \\
			0 & \cdots & k_N^{ \pm}(z)^{-1}
		\end{array}\right) \\ \label{e:Gauss}
		& \times\left(\begin{array}{cccc}
			1 & \cdots & & 0 \\
			-e_1^{ \pm}(z) & \ddots & & \\
			\vdots & \ddots & & \vdots \\
			\cdots & & -e_{N-1}^{ \pm}(z) & 1
		\end{array}\right).
	\end{align}
	
	\subsection{Ungraded vs. graded multiplication rule \cite{Zh,GZ}}
	Let us define the actions of the matrix elements $R(z)$ and $L^{ \pm}(z)$ on $V\otimes V$ by
	\begin{gather*}
		R(z)\left(v_{\alpha^{\prime}} \otimes v_{\beta^{\prime}}\right)=R(z)_{\alpha \beta, \alpha^{\prime} \beta^{\prime}}\left(v_\alpha \otimes v_\beta\right),
		\\
		L^{ \pm}(z) v_{\alpha^{\prime}}=L^{ \pm}(z)_{\alpha \alpha^{\prime}} v_\alpha.
	\end{gather*}
where we have adopted the Riemann summation convention that the repeated indices represent summation, i.e., the first equation
means that $ R(z)\left(v_{\alpha^{\prime}} \otimes v_{\beta^{\prime}}\right)=\sum_{\alpha, \beta}R(z)_{\alpha \beta, \alpha^{\prime} \beta^{\prime}}\left(v_\alpha \otimes v_\beta\right)$.

	In the matrix form,  it carries extra signs due to the graded multiplication rule of the tensor
	products
	\begin{equation}
		\begin{aligned}
			& R\left(\frac{z}{w}\right)_{\alpha \beta, \alpha^{\prime \prime} \beta^{\prime \prime}} L^{ \pm}(z)_{\alpha^{\prime \prime} \alpha^{\prime}} L^{ \pm}(w)_{\beta^{\prime \prime} \beta^{\prime}}(-1)^{\left[\alpha^{\prime}\right]\left(\left[\beta^{\prime}\right]+\left[\beta^{\prime \prime}\right]\right)} \\
			& \quad=L^{ \pm}(w)_{\beta \beta^{\prime \prime}} L^{ \pm}(z)_{\alpha \alpha^{\prime \prime}} R\left(\frac{z}{w}\right)_{\alpha^{\prime \prime} \beta^{\prime \prime}, \alpha^{\prime} \beta^{\prime}}(-1)^{[\alpha]\left([\beta]+\left[\beta^{\prime \prime}\right]\right)},
			\\
			& R\left(\frac{z_{+}}{w_{-}}\right)_{\alpha \beta, \alpha^{\prime \prime} \beta^{\prime \prime}} L^{+}(z)_{\alpha^{\prime \prime} \alpha^{\prime}} L^{-}(w)_{\beta^{\prime \prime} \beta^{\prime}}(-1)^{\left[\alpha^{\prime}\right]\left(\left[\beta^{\prime}\right]+\left[\beta^{\prime \prime}\right]\right)} \\
			& \quad=L^{-}(w)_{\beta \beta^{\prime \prime}} L^{+}(z)_{\alpha \alpha^{\prime \prime}} R\left(\frac{z_{-}}{w_{+}}\right)_{\alpha^{\prime \prime} \beta^{\prime \prime}, \alpha^{\prime} \beta^{\prime}}(-1)^{[\alpha]\left([\beta]+\left[\beta^{\prime \prime}\right]\right)}.
		\end{aligned}
	\end{equation}
	We introduce the matrix $\theta$ as follows:
	$$
	\theta_{\alpha \beta, \alpha^{\prime} \beta^{\prime}}=(-1)^{[\alpha][\beta]} \delta_{\alpha \alpha^{\prime}} \delta_{\beta \beta^{\prime}}.
	$$
	With the help of this matrix $\theta$,  the usual matrix equations become
	\begin{align}
		\label{RLL1.1}
		&R\left(\frac{z}{w}\right) L_1^{ \pm}(z) \theta L_2^{ \pm}(w) \theta=\theta L_2^{ \pm}(w) \theta L_1^{ \pm}(z) R\left(\frac{z}{w}\right),\\
		\label{RLL1.2}
		&R\left(\frac{z_{+}}{w_{-}}\right) L_1^{+}(z) \theta L_2^{-}(w) \theta=\theta L_2^{-}(w) \theta L_1^{+}(z) R\left(\frac{z_{-}}{w_{+}}\right).
	\end{align}
	
	Note that the tensor products above don't necessarily care about gradings.
	
	It is easy to deduce the following matrix equations from the $RLL$ relations:
	\begin{align}	
	\label{RLL2.1}
	& R_{21}\left(\frac{z}{w}\right) \theta L_2^{ \pm}(z) \theta L_1^{ \pm}(w)=L_1^{ \pm}(w) \theta L_2^{ \pm}(z) \theta R_{21}\left(\frac{z}{w}\right), \\ 	
	& R_{21}\left(\frac{z_{-}}{w_{+}}\right) \theta L_2^{-}(z) \theta L_1^{+}(w)=L_1^{+}(w) \theta L_2^{-}(z) \theta R_{21}\left(\frac{z_{+}}{w_{-}}\right),\\
	& \theta L_2^{ \pm}(z)^{-1} \theta L_1^{ \pm}(w)^{-1} R_{21}\left(\frac{z}{w}\right)=R_{21}\left(\frac{z}{w}\right) L_1^{ \pm}(w)^{-1} \theta L_2^{ \pm}(z)^{-1} \theta ,\\
	\label{RLL2.4}
	& \theta L_2^{+}(z)^{-1} \theta L_1^{-}(w)^{-1} R_{21}\left(\frac{z_{+}}{w_{-}}\right)=R_{21}\left(\frac{z_{-}}{w_{+}}\right) L_1^{-}(w)^{-1} \theta L_2^{+}(z)^{-1} \theta, \\
	\label{RLL2.5}
	& L_1^{ \pm}(w)^{-1} R_{21}\left(\frac{z}{w}\right) \theta L_2^{ \pm}(z) \theta=\theta L_2^{ \pm}(z) \theta R_{21}\left(\frac{z}{w}\right) L_1^{ \pm}(w)^{-1} ,\\
	& L_1^{-}(w)^{-1} R_{21}\left(\frac{z_{+}}{w_{-}}\right) \theta L_2^{+}(z) \theta=\theta L_2^{+}(z) \theta R_{21}\left(\frac{z_{-}}{w_{+}}\right) L_1^{-}(w)^{-1}, \\
	\label{RLL2.7}
	& L_1^{+}(w)^{-1} R_{21}\left(\frac{z_{-}}{w_{+}}\right) \theta L_2^{-}(z) \theta=\theta L_2^{-}(z) \theta R_{21}\left(\frac{z_{+}}{w_{-}}\right) L_1^{+}(w)^{-1},
\end{align}
where $R_{21}\left(\frac{z}{w}\right)=P R_{12}\left(\frac{z}{w}\right) P$.
	\section{Drinfeld realization of $U_{p, q}(\widehat{\mathfrak{gl}}(m|n))$}
	\subsection{Drinfeld current realization of $U_{p, q}(\widehat{\mathfrak{gl}}(m|n))$}
Let $m, n \geq 1$ be fixed positive integers and let $I=\{1,2, \ldots, m+n-1\}$.
	Following \cite{Z}, let $V$ be a $(m+n)$-dimensional graded vector space with the even basis vectors $\left\{v^1, v^2, \ldots, v^m\right\}$ and the odd basis vectors $\left\{v^{m+1}, v^{m+2}, \ldots, v^{m+n}\right\}$. Then the basic $R$-matrix of  $U_{p, q}(\mathfrak{gl}(m|n))$ can be written as
	\begin{equation}
		\begin{aligned}
			R&=(-1)^{[i]}\Bigl(\sum_{i \leq m} E_{i i} \otimes E_{i i}+p q \sum_{i>m} E_{i i} \otimes E_{i i}\Bigr)\\
			&\quad +(-1)^{[i][j]}\Bigl(p \sum_{i>j} E_{i j} \otimes E_{j i}+q \sum_{i<j} E_{i j} \otimes E_{j i}\Bigr)	\\
			&\quad+(1-p q) \sum_{i<j} E_{j j} \otimes E_{i i}.	
		\end{aligned}
	\end{equation}
	
	 The Cartan matrix of the Lie superalgebra $\mathfrak{gl}(m|n)$
is the $|I| \times|I|$ matrix $C=\left(a_{i j}\right)$ given by $a_{i j}=\left(1+(-1)^{\delta_{i, m}}\right) \delta_{i, j}-\delta_{i, j+1}-(-1)^{\delta_{i, m}} \delta_{i, j-1}$. 
	\begin{defi} \cite{CY}
		The two-parameter quantum superalgebra $U_{p,q^{-1}}(\mathfrak{gl}(m|n))$ is the $\mathbb{Z}_2$-graded associative algebra generated by $e_i, f_i, a_i^{ \pm 1}, b_i^{ \pm 1}$ $(i \in I)$ satisfying the following relations:
		\begin{align*}
			& a_i a_i^{-1}=b_j b_j^{-1}=1 ; \\
			& a_i^{ \pm 1} a_j^{ \pm 1}=a_j^{ \pm 1} a_i^{ \pm 1},\quad b_i^{ \pm 1} b_j^{ \pm 1}=b_j^{ \pm 1} b_i^{ \pm 1},\quad a_i^{ \pm 1} b_j^{ \pm 1}=b_j^{ \pm 1} a_i^{ \pm 1} . \\
			& \text { If } \ 1 \leq i<m, \text { or } \ i=m, j=m-1, \\
			& a_i e_j=p^{\left\langle\varepsilon_i, \alpha_j\right\rangle} e_j w_i, \quad a_i f_j=p^{-\left\langle\varepsilon_i, \alpha_j\right\rangle}  f_j w_i; \\
			& b_i e_j=q^{-\left\langle\varepsilon_i, \alpha_j\right\rangle} e_j w_i, \quad b_i f_j=q^{\left\langle\varepsilon_i, \alpha_j\right\rangle}  f_j w_i. \\
			& \text { If } \ m+1 \leq i \leq m+n-1, \text { or } \ i=m, j=m+1, \\
			& a_i e_j=p^{\left\langle\varepsilon_{i+1}, \alpha_j\right\rangle} e_j w_i, \quad a_i f_j=p^{-\left\langle\varepsilon_{i+1}, \alpha_j\right\rangle}f_j w_i ; \\
			& b_i e_j=q^{-\left\langle\varepsilon_{i+1}, \alpha_j\right\rangle} e_j w_i, \quad b_i f_j=q^{\left\langle\varepsilon_{i+1}, \alpha_j\right\rangle}f_j w_i. \\
			& a_m e_m=e_m a_m, \quad a_m f_m=f_m a_m; \\
			& b_m e_m=e_m b_m, \quad b_m f_m=f_m b_m. \\
			& e_i f_j-(-1)^{p\left(e_i\right) p\left(f_j\right)} f_j e_i=\delta_{i, j} \frac{a_ib_{i+1}-a_{i+1}b_i}{p-q^{-1}},  \quad(i, j \in I) ; \\
			& e_i e_j-e_j e_i=0, \quad f_i f_j-f_j f_i=0, \quad|i-j| \geq 2 ;
		\end{align*}
		$$
		\begin{aligned}
			& e_i^2 e_{i+1}-(p+q^{-1}) e_i e_{i+1} e_i+pq^{-1}e_{i+1} e_i^2=0 \quad(i \neq m) ; \\
			& e_i^2 e_{i-1}-\left(p^{-1}+q\right) e_i e_{i-1} e_i+p^{-1}q e_{i-1} e_i^2=0 \quad(i \neq m) ; \\
			& f_i^2 f_{i+1}-\left(p^{-1}+q\right) f_i f_{i+1} f_i+p^{-1} q f_{i+1} f_i^2=0 \quad(i \neq m) ; \\
			& f_i^2 f_{i-1}-(p+q^{-1}) f_i f_{i-1} f_i+pq^{-1} f_{i-1} f_i^2=0 \quad(i \neq m) ; \\
			& e_m^2=f_m^2=0 ; \\
			& e_m e_{m-1} e_m e_{m+1}+pq^{-1} e_m e_{m+1} e_m e_{m-1}+e_{m-1} e_m e_{m+1} e_m \\
			& \quad+pq^{-1} e_{m+1} e_m e_{m-1} e_m-(p+q^{-1}) e_m e_{m-1} e_{m+1} e_m=0 ; \\
			& f_m f_{m-1} f_m f_{m+1}+p^{-1}q f_m f_{m+1} f_m f_{m-1}+f_{m-1} f_m f_{m+1} f_m \\
			& \quad+p^{-1}q f_{m+1} f_m f_{m-1}f_m
			-\left(p^{-1}+q\right) f_m f_{m-1} f_{m+1} f_m=0.
		\end{aligned}
		$$
		Here the $\mathbb{Z}_2$-grading function (parity) $p(x)$ is given by $p(a_i^{ \pm 1})=p(b_i^{ \pm 1})=0$ $(i \in I)$, $p\left(e_i\right)=p\left(f_i\right)=0$ $(i \neq m)$ and $p\left(e_m\right)=p\left(f_m\right)=1$.
	\end{defi}

	Using the Yang-Baxterization or Jimbo's method \cite{J}, we have the following quantum affine  $R$-matrix given by
	\begin{align}\notag
		&\tilde{R}_{12}(z)=  \sum_{i=1}^m E_{i i} \otimes E_{i i}+\frac{q{-}z p^{-1}}{z q{-}p^{-1}} \sum_{i=m+1}^{m+n} E_{i i} \otimes E_{i i}+\frac{(z{-}1) q p^{-1}}{z q{-}p^{-1}} \sum_{i<j} E_{i i} \otimes E_{j j}\\ \label{Basic R}
		& +\frac{z{-}1}{z q{-}p^{-1}} \sum_{i>j} E_{i i} \otimes E_{j j}+\frac{z(q{-}p^{-1})}{q{-}p^{-1}}(z\sum_{i<j} E_{i j} \otimes E_{j i}+\sum_{i>j} E_{i j} \otimes E_{j i}),
	\end{align}
	and
	$$
	R_{12}(z)_{\alpha \beta}^{\alpha^{\prime} \beta^{\prime}}=(-1)^{[\alpha][\beta]} \tilde{R}(z)_{\alpha \beta}^{\alpha^{\prime} \beta^{\prime}}.
	$$
	For the $N=m+n=2$ case, 
$L^{ \pm}(z)$ and $L^{ \pm}(z)^{-1}$ can be written explicitly as
	\begin{equation}
		L^{ \pm}(z)=\left(\begin{array}{cc}
			k_1^{ \pm}(z) & k_1^{ \pm}(z) f_1^{ \pm}(z) \\
			e_1^{ \pm}(z) k_1^{ \pm}(z) & k_2^{ \pm}(z)+e_1^{ \pm}(z) k_1^{ \pm}(z) f_1^{ \pm}(z)
		\end{array}\right),
	\end{equation}
	
	\begin{equation}
		L^{ \pm}(z)^{-1}=\left(\begin{array}{cc}
			k_1^{ \pm}(z)^{-1}+f_1^{ \pm}(z) k_2^{ \pm}(z)^{-1} e_1^{ \pm}(z) & -f_1^{ \pm}(z) k_2^{ \pm}(z)^{-1} \\
			-k_2^{ \pm}(z)^{-1} e_1^{ \pm}(z) & k_2^{ \pm}(z)^{-1}
		\end{array}\right) .
	\end{equation}
	There are three different internal block types.
	For $m=2, n=0$,
	$${\text{ Type }}1: \ {R_{{\text{12}}}}\left( {\frac{z}{w}} \right) = \left( {\begin{array}{*{20}{c}}
			1&0&0&0 \\
			0&{\frac{{(z - w)q{p^{ - 1}}}}{{zq - w{p^{ - 1}}}}}&{\frac{{z\left( {q - {p^{ - 1}}} \right)}}{{zq - w{p^{ - 1}}}}}&0 \\
			0&{\frac{{w\left( {q - {p^{ - 1}}} \right)}}{{zq - w{p^{ - 1}}}}}&{\frac{{z - w}}{{zq - w{p^{ - 1}}}}}&0 \\
			0&0&0&1
	\end{array}} \right){\text{. }}$$
	
	For $m=1, n=1$,
	\begin{equation}{\text{ Type 2}}: \ {R_{{\text{12}}}}\left( {\frac{z}{w}} \right) = \left( {\begin{array}{*{20}{c}}
				1&0&0&0 \\
				0&{\frac{{(z - w)q{p^{ - 1}}}}{{zq - w{p^{ - 1}}}}}&{\frac{{z\left( {q - {p^{ - 1}}} \right)}}{{zq - w{p^{ - 1}}}}}&0 \\
				0&{\frac{{w\left( {q - {p^{ - 1}}} \right)}}{{zq - w{p^{ - 1}}}}}&{\frac{{z - w}}{{zq - w{p^{ - 1}}}}}&0 \\
				0&0&0&{ - \frac{{wq - z{p^{ - 1}}}}{{zq - w{p^{ - 1}}}}}
		\end{array}} \right){\text{.}}\end{equation}
	
	For $m=0, n=2$,
	$${\text{ Type 3}}: \ {R_{{\text{12}}}}\left( {\frac{z}{w}} \right) = \left( {\begin{array}{*{20}{c}}
			{ - \frac{{wq - z{p^{ - 1}}}}{{zq - w{p^{ - 1}}}}}&0&0&0 \\
			0&{\frac{{(z - w)q{p^{ - 1}}}}{{zq - w{p^{ - 1}}}}}&{-\frac{{z\left( {q - {p^{ - 1}}} \right)}}{{zq - w{p^{ - 1}}}}}&0 \\
			0&{-\frac{{w\left( {q - {p^{ - 1}}} \right)}}{{zq - w{p^{ - 1}}}}}&{\frac{{z - w}}{{zq - w{p^{ - 1}}}}}&0 \\
			0&0&0&{ - \frac{{wq - z{p^{ - 1}}}}{{zq - w{p^{ - 1}}}}}
	\end{array}} \right){\text{. }}$$
	Let's focus on type one: $m=1,$ $n=1$.  Let
	$$
	\theta=\left(\begin{array}{cccc}
		1 & 0 & 0 & 0 \\
		0 & 1 & 0 & 0 \\
		0 & 0 & 1 & 0 \\
		0 & 0 & 0 & -1
	\end{array}\right).
	$$
	By \eqref{RLL2.1}-\eqref{RLL2.7} and \eqref{e:Xpm1}-\eqref{e:Xpm2}, we obtain the following relations among the components
	\begin{gather}
		k_i^{ \pm}(z) k_j^{ \pm}(w)=k_j^{ \pm}(w) k_i^{ \pm}(z), \quad i=1,2,\\
		\frac{z_{ \pm}-w_{\mp}}{z_{ \pm} q-w_{\mp} p^{-1}} k_2^{\mp}(w)^{-1} k_1^{ \pm}(z)=k_1^{ \pm}(z) k_2^{\mp}(w)^{-1} \frac{z_{\mp}-w_{ \pm}}{z_{\mp} q-w_{ \pm} p^{-1}},\\
		k_1^{+}(z) k_1^{-}(w)=k_1^{-}(w) k_1^{+}(z),\\
		\frac{w_{-} q-p^{-1} z_{+}}{z_{+} q-w_{-} p^{-1}} k_2^{+}(z) k_2^{-}(w)=\frac{w_{+} q-p^{-1} z_{-}}{z_{-} q-w_{+} p^{-1}} k_2^{-}(w) k_2^{+}(z),\\
k_1^{ \pm}(z)^{\epsilon} X_1^{\epsilon}(w) k_1^{ \pm}(z)^{-\epsilon}=\frac{z_{\pm\epsilon} q-w p^{-1}}{\left(z_{\pm\epsilon}-w\right) q p^{-1}} X_1^{\epsilon}(w),\\
\label{X1X1 rel1 type2}
		X_1^{\pm}(z) X_1^{\pm}(w)+X_1^{\pm}(w) X_1^{\pm}(z)=0,\\
		k_2^{ \pm}(z) X_1^{\epsilon}(w) k_2^{ \pm}(z)^{-\epsilon}=\frac{z_{ \pm \epsilon} p^{-1}-wq }{\left(z_{ \pm \epsilon}-w\right) q p^{-1}} X_1^{+}(w),\\
		\left\{X_1^{+}(z), X_1^{-}(w)\right\}=\left(p-q^{-1}\right)\Bigl\{\delta\left(z w^{-1} q^{-c}\right) k_1^{+}\left(w_{+}\right)^{-1} k_2^{+}\left(w_{+}\right)\\
		\qquad\qquad\qquad\qquad-\delta\left(z^{-1} w q^{-c}\right) k_1^{-}\left(z_{+}\right)^{-1} k_2^{-}\left(z_{+}\right)\Bigr\}.\notag
	\end{gather}
	Similarly for $m=2,$  $n=0$, 
the main relations are as follows. 
	\begin{gather}
		k_i^{ \pm}(z) k_j^{ \pm}(w)=k_j^{ \pm}(w) k_i^{ \pm}(z),\\
		k_i^{+}(z) k_i^{-}(w)=k_i^{-}(w) k_i^{+}(z), \quad i=1,2,\\
		\frac{z_{ \pm}-w_{\mp}}{z_{ \pm} p-w_{\mp} q^{-1}} k_2^{\mp}(w)^{-1} k_1^{ \pm}(z)=\frac{z_{\mp}-w_{ \pm}}{z_{\mp} p-w_{ \pm} q^{-1}} k_1^{ \pm}(z) k_2^{\mp}(w)^{-1},\\
k_1^{ \pm}(z)^{-1} X_1^{\epsilon}(w) k_1^{ \pm}(z)=\frac{z p-w_{\mp\epsilon} q^{-1}}{z-w_{\mp\epsilon}} X_1^{\epsilon}(w),\\
		\label{X1X1 rel1 type1}
		\left(z q^{-1}-w p\right) X_1^{-}(z) X_1^{-}(w)=\left(z p-w q^{-1}\right) X_1^{-}(w) X_1^{-}(z),\\
		\label{X1X1 rel2 type1}
		\left(z p-w q^{-1}\right) X_1^{+}(z) X_1^{+}(w)=\left(z q^{-1}-w p\right) X_1^{+}(w) X_1^{+}(z),\\
k_2^{ \pm}(w)^{-1} X_1^{\epsilon}(z) k_2^{ \pm}(w)=(\frac{z_{ \mp\epsilon}-w}{z_{ \mp\epsilon} p-w q^{-1}})^{\epsilon} X_1^{\epsilon}(z),\\
		\left[X_1^{+}(z), X_1^{-}(w)\right]=-\left(p-q^{-1}\right)\Bigl\{\delta\left(z^{-1} w q^c\right) k_1^{+}\left(w_{+}\right)^{-1} k_2^{+}\left(w_{+}\right)\\
		\qquad\qquad\qquad\qquad-\delta\left(z^{-1} w q^{-c}\right) k_1^{-}\left(z_{+}\right)^{-1} k_2^{-}\left(z_{+}\right)\Bigr\}.\notag
	\end{gather}
	For $m=0,$  $n=2$, the relations are as follows.
	\begin{gather}
		k_i^{ \pm}(z) k_j^{ \pm}(w)=k_j^{ \pm}(w) k_i^{ \pm}(z), \quad i=1,2\\
		\frac{w_{\mp} q-p^{-1} z_{\pm}}{z_{\pm} q-w_{\mp} p^{-1}} k_i^{\pm}(z) k_i^{\mp}(w)=\frac{w_{\pm} q-p^{-1} z_{\mp}}{z_{\mp} q-w_{\pm} p^{-1}} k_i^{\mp}(w) k_i^{\pm}(z), \quad i=1,2\\
		\frac{z_{ \pm}-w_{\mp}}{z_{ \pm} p-w_{\mp} q^{-1}} k_2^{\mp}(w)^{-1} k_1^{ \pm}(z)=\frac{z_{\mp}-w_{ \pm}}{z_{\mp} p-w_{ \pm} q^{-1}} k_1^{ \pm}(z) k_2^{\mp}(w)^{-1},\\
k_1^{ \pm}(z)^{\epsilon} X_1^{\epsilon}(w) k_1^{ \pm}(z)^{-\epsilon}=\frac{z_{\pm\epsilon} p^{-1}-w q}{\left(z_{\pm\epsilon}-w\right) q p^{-1}} X_1^{\epsilon}(w),\\
		\label{X1X1 rel1 type3}
		\left(z p-w q^{-1}\right) X_1^{-}(z) X_1^{-}(w)=\left(z q^{-1}-w p\right) X_1^{-}(w) X_1^{-}(z),\\
		\label{X1X1 rel2 type3}
		\left(z q^{-1}-w p\right) X_1^{+}(z) X_1^{+}(w)=\left(z p-w q^{-1}\right) X_1^{+}(w) X_1^{+}(z),\\
k_2^{ \pm}(z)^{\epsilon} X_1^{\epsilon}(w) k_2^{ \pm}(z)^{-\epsilon}=\frac{z_{\pm\epsilon} q-w p^{-1}}{\left(z_{\pm\epsilon}-w\right) q p^{-1}} X_1^{\epsilon}(w),\\
		\left[X_1^{+}(z), X_1^{-}(w)\right]=-\left(p-q^{-1}\right)\Bigl\{\delta\left(z^{-1} w q^c\right) k_1^{+}\left(w_{+}\right)^{-1} k_2^{+}\left(w_{+}\right)\\
		\qquad\qquad\qquad\quad	-\delta\left(z^{-1} w q^{-c}\right) k_1^{-}\left(z_{+}\right)^{-1} k_2^{-}\left(z_{+}\right)\Bigr\}.\notag
	\end{gather}
	
	We will extend the results from $N=2$ to $N=3$, and then to general $N$.
	
	$N=3$: we take $m=2, n=1$ as an example. The calculations of $R$-matrices for the other cases are similar. It follows from
\eqref{Basic R} that the $R$-matrix for the case $m=2, n=1$ is given as follows. 
	\begin{equation}
		\begin{split}	
			&{R_{21}}\Bigl(\frac{z}{w}\Bigr) =\\
			& \left( {\begin{array}{*{20}{c}}
					1&0&0&0&0&0&0&0&0 \\
					0&{\frac{{z - w}}{{zq - w{p^{ - 1}}}}}&0&{\frac{{w\left( {q - {p^{ - 1}}} \right)}}{{zq - w{p^{ - 1}}}}}&0&0&0&0&0 \\
					0&0&{\frac{{z - w}}{{zq - w{p^{ - 1}}}}}&0&0&0&{\frac{{w\left( {q - {p^{ - 1}}} \right)}}{{zq - w{p^{ - 1}}}}}&0&0 \\
					0&{\frac{{z\left( {q - {p^{ - 1}}} \right)}}{{zq - w{p^{ - 1}}}}}&0&{\frac{{(z - w)q{p^{ - 1}}}}{{zq - w{p^{ - 1}}}}}&0&0&0&0&0 \\
					0&0&0&0&1&0&0&0&0 \\
					0&0&0&0&0&{\frac{{z - w}}{{zq - w{p^{ - 1}}}}}&0&{\frac{{w\left( {q - {p^{ - 1}}} \right)}}{{zq - w{p^{ - 1}}}}}&0 \\
					0&0&{\frac{{z\left( {q - {p^{ - 1}}} \right)}}{{zq - w{p^{ - 1}}}}}&0&0&0&{\frac{{(z - w)q{p^{ - 1}}}}{{zq - w{p^{ - 1}}}}}&0&0 \\
					0&0&0&0&0&{\frac{{z\left( {q - {p^{ - 1}}} \right)}}{{zq - w{p^{ - 1}}}}}&0&{\frac{{(z - w)q{p^{ - 1}}}}{{zq - w{p^{ - 1}}}}}&0 \\
					0&0&0&0&0&0&0&0&{ - \frac{{wq - z{p^{ - 1}}}}{{zq - w{p^{ - 1}}}}}
			\end{array}} \right).
		\end{split}
	\end{equation}
	We first divide the $N=3$  case into two $N=2$  cases by decomposing the $R$-matrix into a type 1 $R$-matrix $(1 \leqslant i, j \leqslant 2)$ and type 2 $R$-matrix $(2 \leqslant i, j \leqslant 3)$. Their matrices $L^{ \pm}(z)$ can be written respectively as
	\begin{gather}
		L^{ \pm}(z)=\left(\begin{array}{cc}
			1 & 0 \\
			e_1^{ \pm}(z) & 1
		\end{array}\right)\left(\begin{array}{cc}
			k_1^{ \pm}(z) & 0 \\
			0 & k_2^{ \pm}(z)
		\end{array}\right)\left(\begin{array}{cc}
			1 & f_1^{ \pm}(z) \\
			0 & 1
		\end{array}\right),
	\end{gather}
	\begin{gather}
		L^{ \pm}(z)=\left(\begin{array}{cc}
			1 & 0 \\
			e_2^{ \pm}(z) & 1
		\end{array}\right)\left(\begin{array}{cc}
			k_2^{ \pm}(z) & 0 \\
			0 & k_3^{ \pm}(z)
		\end{array}\right)\left(\begin{array}{cc}
			1 & f_2^{ \pm}(z) \\
			0 & 1
		\end{array}\right).
	\end{gather}
	Relations among $k_1^{ \pm}(z), k_2^{ \pm}(z), e_1^{ \pm}(z)$, $f_1^{ \pm}(z)$  and among $k_2^{ \pm}(z), k_3^{ \pm}(z), e_2^{ \pm}(z), f_2^{ \pm}(z)$
 have already been computed above, we only need to get relations between $k_1^{ \pm}(z)$, $e_1^{ \pm}(z)$, $f_1^{ \pm}(z)$ and $k_3^{ \pm}(z), e_2^{ \pm}(z)$, $f_2^{ \pm}(z)$. For this, we write
 $L^{ \pm}(z)$ and $L^{ \pm}(z)^{-1}$ in the following forms
	\begin{gather}
		L^{ \pm}(z)=\left(\begin{array}{ccc}
			k_1^{ \pm}(z) & k_1^{ \pm}(z) f_1^{ \pm}(z) & * \\
			e_1^{ \pm}(z) k_1^{ \pm}(z) & * & * \\
			e_{3,1}^{ \pm}(z) k_1^{ \pm}(z) & * & *
		\end{array}\right),
	\end{gather}
	\begin{equation}
		\begin{split}
			&L^{ \pm}(w)^{-1}=\\
			&\left(\begin{array}{ccc}
				* & * & * \\
				* & * & -f_2^{ \pm}(w) k_3^{ \pm}(w)^{-1} \\
				k_3^{ \pm}(w)^{-1}\left[e_2^{ \pm}(w) e_1^{ \pm}(w)-e_{3,1}^{ \pm}(w)\right] & -k_3^{ \pm}(w)^{-1} e_2^{ \pm}(w) & k_3^{ \pm}(w)^{-1}
			\end{array}\right).
		\end{split}
	\end{equation}
	where $*$ represent some elements in the $U(R)$. 
	For convenience, we rewrite (\ref{RLL2.5})-(\ref{RLL2.7}) explicitly as
	\begin{gather}
		\label{RLL rel1 N=3}
		\left(L_1^{ \pm}(w)^{-1}\right)_{j_1}^{i_1} R_{21}\left(\frac{z}{w}\right)_{k_1 j_2}^{j_1 i_2} L_2^{ \pm}(z)_{k_2}^{j_2}=L_2^{ \pm}(z)_{j_2}^{i_2} R_{21}\left(\frac{z}{w}\right)_{j_1 k_2}^{i_1 j_2}\left(L_1^{ \pm}(w)^{-1}\right)_{k_1}^{j_1},
	\end{gather}
	\begin{gather}
		\label{RLL rel2 N=3}
		\left(L_1^{\mp}(w)^{-1}\right)_{j_1}^{i_1} R_{21}\left(\frac{z_{ \pm}}{w_{\mp}}\right)_{k_1 j_2}^{j_1 i_2} L_2^{ \pm}(z)_{k_2}^{j_2}=L_2^{ \pm}(z)_{j_2}^{i_2} R_{21}\left(\frac{z_{\mp}}{w_{ \pm}}\right)_{j_1 k_2}^{i_1 j_2}\left(L_1^{\mp}(w)^{-1}\right)_{k_1}^{j_1},
	\end{gather}
where $i_1, i_2, k_1, k_2$ are free indices, summations over $j_1, j_2$ are assumed. By taking special values of $i_1, k_1, i_2, k_2$, we obtain the following relations:
	\begin{gather}
		\label{k1k3 rel1}
		k_1^{ \pm}(z) k_3^{ \pm}(w)=k_3^{ \pm}(w) k_1^{ \pm}(z),\\
		\label{k1k3 rel2}
		\frac{z_{ \pm}-w_{\mp}}{z_{ \pm} q-w_{\mp} p^{-1}} k_3^{\mp}(w)^{-1} k_1^{ \pm}(z)=k_1^{ \pm}(z) k_3^{\mp}(w)^{-1} \frac{z_{\mp}-w_{ \pm}}{z_{\mp} q-w_{ \pm} p^{-1}},\\
		\label{k3e1 rel1}
e_1^{\epsilon}(z) k_3^{\epsilon'}(w)=k_3^{\epsilon'}(w) e_1^{ \epsilon}(z),\\
k_3^{\epsilon}(w) f_1^{\epsilon'}(z)=f_1^{\epsilon'}(z) k_3^{ \epsilon}(w),\\
k_1^{\epsilon}(w) f_2^{\epsilon'}(z)=f_2^{\epsilon'}(z) k_1^{ \epsilon}(w),\\
e_2^{ \epsilon}(w) k_1^{ \epsilon'}(z)=k_1^{ \epsilon'}(z) e_2^{ \epsilon}(w),\\
e_2^{ \epsilon}(w) f_1^{ \epsilon'}(z)=f_1^{ \epsilon'}(z) e_2^{ \epsilon}(w),\\
f_2^{ \epsilon}(w) e_1^{ \epsilon'}(z)=e_1^{ \epsilon'}(z) f_2^{ \epsilon}(w),
	\end{gather}
where $\epsilon, \epsilon'\in\{+, -\}$.
	One can prove that these relations are true for all situations in the $N=3$ case.
	Then, let $i_1=3, k_1=1, i_2=2, k_2=1$, we have
	\begin{align}\notag
		& -\frac{(z-w)qp^{-1}}{z q-w p^{-1}} e_1^{ \pm}(z) k_1^{ \pm}(z) k_3^{ \pm}(w)^{-1} e_2^{ \pm}(w) \\  \label{X1X2 pre rel1}
			& =\frac{w\left(q-p^{-1}\right)}{z q-w p^{-1}} k_3^{ \pm}(w)^{-1}\left[-e_{3,1}^{ \pm}(w)+e_2^{ \pm}(w) e_1^{ \pm}(w)\right] k_1^{ \pm}(z) \\ \notag
			& \quad-d\left(\frac{z}{w}\right) k_3^{ \pm}(w)^{-1} e_2^{ \pm}(w) e_1^{ \pm}(z) k_1^{ \pm}(z)+\frac{z\left(q-p^{-1}\right)}{z q-w p^{-1}} k_3^{ \pm}(w)^{-1} e_{3,1}^{ \pm}(z) k_1^{ \pm}(z),
	\end{align}
	\begin{align}\notag
		& -\frac{(z_{\mp}-w_{ \pm})qp^{-1}}{z_{\mp} q-w_{ \pm} p^{-1}} e_1^{ \pm}(z) k_1^{ \pm}(z) k_3^{\mp}(w)^{-1} e_2^{\mp}(w) \\ \label{X1X2 pre rel2}
			& =\frac{w_{\mp}\left(q-p^{-1}\right)}{z_{ \pm} q-w_{\mp} p^{-1}} k_3^{\mp}(w)^{-1}\left[-e_{3,1}^{\mp}(w)+e_2^{\mp}(w) e_1^{\mp}(w)\right] k_1^{ \pm}(z) \\ \notag
			& \quad-d\left(\frac{z_{ \pm}}{w_{\mp}}\right) k_3^{\mp}(w)^{-1} e_2^{\mp}(w) e_1^{ \pm}(z) k_1^{ \pm}(z)+\frac{z_{ \pm}\left(q-p^{-1}\right)}{z_{ \pm} q-w_{\mp} p^{-1}} k_3^{\mp}(w)^{-1} e_{3,1}^{ \pm}(z) k_1^{ \pm}(z).
		\end{align}
	Here $d(z / w)=1$  for the case $m=2$ or $m=3$, $d(z / w)=\left(w q-z p^{-1}\right) /\left(z q-w p^{-1}\right)$ for the case $m=1$ or $m=0$.
	Multiplying on the left by $k_3^{ \pm}(w)$ and on the right by $k_1^{ \pm}(z)^{-1}$ on both sides of equation (\ref{X1X2 pre rel1}), we get that
		\begin{align}
			& -\frac{(z{-}w) q p^{-1}}{z q{-}w p^{-1}} k_3^{ \pm}(w) e_1^{ \pm}(z) k_1^{ \pm}(z) k_3^{ \pm}(w)^{-1} e_2^{ \pm}(w) k_1^{ \pm}(z)^{-1} \\ \notag
			= & \frac{w\left(q{-}p^{-1}\right)}{z q{-}w p^{-1}}\left[-e_{3,1}^{ \pm}(w)+e_2^{ \pm}(w) e_1^{ \pm}(w)\right]-d\left(\frac{z}{w}\right) e_2^{ \pm}(w) e_1^{ \pm}(z)+\frac{z\left(q{-}p^{-1}\right)}{z q{-}w p^{-1}} e_{3,1}^{ \pm}(z),
		\end{align}
	then with the help of (\ref{k1k3 rel1}) and (\ref{k3e1 rel1}) and after some simplifications, we can get
	$$
	\begin{aligned}
		\left(q{-}p^{-1}\right)\Bigl[z e_{3,1}^{ \pm}(z)&-w e_{3,1}^{ \pm}(w)\Bigr]=  \left(z q-w p^{-1}\right) d(z / w) e_2^{ \pm}(w) e_1^{ \pm}(z)\\
		& -w\left(q-p^{-1}\right) e_2^{ \pm}(w) e_1^{ \pm}(w) -(z-w)qp^{-1} e_1^{ \pm}(z) e_2^{ \pm}(w).
	\end{aligned}
	$$
	Replacing the spectral parameters from $z, w$ to $z_{-}, w_{-}$ and $z_{+}, w_{+}$ respectively, we then have the following
	relations: 
	$$
	\begin{aligned}
		\left(q{-}p^{-1}\right)\Bigl[z e_{3,1}^{\pm}\left(z_{\mp}\right)&-w e_{3,1}^{\pm}\left(w_{\mp}\right)\Bigr]= \left(z q-w p^{-1}\right) d(z / w) e_2^{\pm}\left(w_{\mp}\right) e_1^{\pm}\left(z_{\mp}\right)\\
		& -w\left(q{-}p^{-1}\right) e_2^{\pm}\left(w_{\mp}\right)  e_1^{+}\left(w_{\mp}\right)-(z-w)qp^{-1} e_1^{\pm}\left(z_{mp}\right) e_2^{+}\left(w_{-}\right).
	\end{aligned}
	$$
	Similarly, multiplying from both sides of (\ref{X1X2 pre rel2}) on the left by $k_1^{\mp}(w)$ and on the right by $k_1^{ \pm}(z)^{-1}$,  with the help of (\ref{k1k3 rel2}) and (\ref{k3e1 rel1}), we have that
	$$
	\begin{aligned}
		-\frac{(z_{ \pm}-w_{\mp})qp^{-1}}{z_{ \pm} q-w_{\mp} p^{-1}} e_1^{ \pm}(z) e_2^{\mp}(w)&=  \frac{w_{\mp}\left(q-p^{-1}\right)}{z_{ \pm} q-w_{\mp} p^{-1}}\left[-e_{3,1}^{\mp}(w)+e_2^{\mp}(w) e_1^{\mp}(w)\right] \\
		&\quad -d\left(\frac{z_{ \pm}}{w_{\mp}}\right) e_2^{\mp}(w) e_1^{ \pm}(z)+\frac{z_{ \pm}\left(q-p^{-1}\right)}{z_{ \pm} q-w_{\mp} p^{-1}} e_{3,1}^{ \pm}(z).
	\end{aligned}
	$$
	Replacing $z, w$ with $z_{+}, w_{-}$ (resp. $z_{-}, w_{+}$), we have that
	$$
	\begin{aligned}
		\left(q{-}p^{-1}\right)\Bigl[z e_{3,1}^{\pm}\left(z_{\mp}\right)&-w e_{3,1}^{\mp}\left(w_{\pm}\right)\Bigr]=  \left.\left(z q-w p^{-1}\right) d(z / w) e_2^{\mp}\left(w_{\mp}\right) e_1^{\pm}\left(z_{mp}\right)\right) \\
		& -w\left(q-p^{-1}\right) e_2^{\mp}\left(w_{\pm}\right) e_1^{\mp}\left(w_{\pm}\right)-(z-w)qp^{-1} e_1^{\pm}\left(z_{\mp}\right) e_2^{\mp}\left(w_{\pm}\right),
	\end{aligned}
	$$
	Canceling the terms $e_{3,1}^{ \pm}\left(z_{\mp}\right)$, we then have
	$$
	\begin{gathered}
		(z q{-}w p^{-1}) d(z / w)\left[e_2^{-}(w_{+}) e_1^{+}(z_{-}){+}e_2^{+}(w_{-}) e_1^{-}(z_{+}) e_2^{+}(w_{-}) e_1^{+}(z_{-}){-}e_2^{-}(w_{+}) e_1^{-}(z_{+})\right] \\
		=(z{-}w)qp^{-1}\left[e_1^{+}(z_{-}) e_2^{-}(w_{+}){+}e_1^{-}(z_{+}) e_2^{+}(w_{-}){-}e_1^{+}(z_{-}) e_2^{+}(w_{-}){-}e_1^{-}(z_{+}) e_2^{-}(w_{+})\right].
	\end{gathered}
	$$
	Thus
	\begin{gather}
		(z-w)qp^{-1} X_1^{+}(z) X_2^{+}(w)=d(z / w)\left(z q-w p^{-1}\right) X_2^{+}(w) X_1^{+}(z).
	\end{gather}
	Similarly, we can show that $X_1^{-}(z)$ and $X_2^{-}(w)$ satisfy the following commutation relations.
	\begin{gather}
		d(z / w)\left(z q-w p^{-1}\right) X_1^{-}(z) X_2^{-}(w)=(z-w) q p^{-1} X_2^{-}(w) X_1^{-}(z).
	\end{gather}
	Furthermore, we have
	\begin{gather}
		\label{X1X2 rel1}
		(z-w)qp^{-1} X_1^{+}(z) X_2^{+}(w)=\left(z q-w p^{-1}\right) X_2^{+}(w) X_1^{+}(z), \quad m=2, 3\\[2mm]
		(w-z)qp^{-1} X_1^{+}(z) X_2^{+}(w)=\left(w q-z p^{-1}\right) X_2^{+}(w) X_1^{+}(z), \quad m=1, 0,\\
		\left(z q-w p^{-1}\right) X_1^{-}(z) X_2^{-}(w)=(z-w)qp^{-1} X_2^{-}(w) X_1^{-}(z), \quad m=2, 3,\\
		\label{X1X2 rel4}
		\left(w q-z p^{-1}\right) X_1^{-}(z) X_2^{-}(w)=(w-z)qp^{-1} X_2^{-}(w) X_1^{-}(z), \quad m=1, 0 .
	\end{gather}
\subsection{Additional Serre relations} Since there are twice as many group-like generators, it requires more effort to control additional terms in the quantum $R$-matrix. As a result there are more Serre relations from the $RLL$ relations.
	Using (\ref{X1X1 rel1 type2}), 
(\ref{X1X1 rel1 type1})-(\ref{X1X1 rel2 type1}), (\ref{X1X1 rel1 type3})-(\ref{X1X1 rel2 type3}), (\ref{X1X2 rel1})-(\ref{X1X2 rel4}), we obtain
the cubic relations in the $N=3$ cases.  There are four types of $R$-matrices when $N=3$, the
	cubic relations are listed respectively as follows.
	
	Case $m=3$,
	\begin{equation}
		\begin{aligned}
			\label{ser1 m=3}
			& \left\{(p^{-1}q)^{}X_1^{+}\left(z_1\right) X_1^{ +}\left(z_2\right) X_2^{ +}(w)-\left(q+p^{-1}\right) X_1^{ +}\left(z_1\right) X_2^{ +}(w) X_1^{ +}\left(z_2\right)\right. \\
			& \left.\qquad\qquad\quad+\,X_2^{ +}(w) X_1^{ +}\left(z_1\right) X_1^{ +}\left(z_2\right)\right\}+\left\{z_1 \leftrightarrow z_2\right\}=0,
		\end{aligned}
	\end{equation}
	\begin{equation}
		\begin{aligned}
			& \left\{X_1^{-}\left(z_1\right) X_1^{ -}\left(z_2\right) X_2^{ -}(w)-\left(q+p^{-1}\right) X_1^{ -}\left(z_1\right) X_2^{ -}(w) X_1^{ -}\left(z_2\right)\right. \\
			& \left.\qquad+\,(p^{-1}q)^{}X_2^{ -}(w) X_1^{ -}\left(z_1\right) X_1^{ -}\left(z_2\right)\right\}+\left\{z_1 \leftrightarrow z_2\right\}=0,
		\end{aligned}
	\end{equation}
	\begin{equation}
		\begin{aligned}
			& \left\{X_2^{ +}\left(z_1\right) X_2^{ +}\left(z_2\right) X_1^{ +}(w)-\left(q+p^{-1}\right) X_2^{ +}\left(z_1\right) X_1^{+}(w) X_2^{ +}\left(z_2\right)\right. \\
			& \qquad\left.+\,(p^{-1}q)^{}X_1^{ +}(w) X_2^{ +}\left(z_1\right) X_2^{ +}\left(z_2\right)\right\}+\left\{z_1 \leftrightarrow z_2\right\}=0 ,
		\end{aligned}
	\end{equation}
	\begin{equation}
		\begin{aligned}
			& \left\{(p^{-1}q)^{}X_2^{ -}\left(z_1\right) X_2^{ -}\left(z_2\right) X_1^{ -}(w)-\left(q+p^{-1}\right) X_2^{ -}\left(z_1\right) X_1^{ -}(w) X_2^{ -}\left(z_2\right)\right. \\
			& \left.\qquad\qquad\quad+\,X_1^{ -}(w) X_2^{ -}\left(z_1\right) X_2^{ -}\left(z_2\right)\right\}+\left\{z_1 \leftrightarrow z_2\right\}=0 .
		\end{aligned}
	\end{equation}
	
	Case $m=2$,
	\begin{equation}
		\begin{aligned}
			& \left\{(p^{-1}q)X_1^{+}\left(z_1\right) X_1^{+}\left(z_2\right) X_2^{+}(w)-\left(q+p^{-1}\right) X_1^{+}\left(z_1\right) X_2^{+}(w) X_1^{+}\left(z_2\right)\right. \\
			& \left.\qquad\qquad\quad+\,X_2^{+}(w) X_1^{+}\left(z_1\right) X_1^{+}\left(z_2\right)\right\}+\left\{z_1 \leftrightarrow z_2\right\}=0,
		\end{aligned}
	\end{equation}
	
\begin{align}\notag
			& \left\{( z _ { 1 } p ^ { - 1 } - z _ { 2 } q ) \left[X_2^{+}\left(z_1\right) X_2^{+}\left(z_2\right) X_1^{+}(w)-\left(q+p^{-1}\right) X_2^{+}\left(z_1\right) X_1^{+}(w) X_2^{+}\left(z_2\right)\right.\right. \\ \label{ser2 m=2}
			& \left.\left.\quad\qquad\qquad+\,(p^{-1}q)X_1^{+}(w) X_2^{+}\left(z_1\right) X_2^{+}\left(z_2\right)\right]\right\}+\left\{z_1 \leftrightarrow z_2\right\}=0,
\end{align}
	\begin{equation}
		\begin{aligned}
			& \left\{X_1^{-}\left(z_1\right) X_1^{-}\left(z_2\right) X_2^{-}(w)-\left(q+p^{-1}\right) X_1^{-}\left(z_1\right) X_2^{-}(w) X_1^{-}\left(z_2\right)\right. \\
			& \left.\qquad+\,(p^{-1}q)^{}X_2^{-}(w) X_1^{-}\left(z_1\right) X_1^{-}\left(z_2\right)\right\}+\left\{z_1 \leftrightarrow z_2\right\}=0,
		\end{aligned}
	\end{equation}
\begin{align}\notag
			& \left\{( z _ { 1 } q - z _ { 2 } p ^ { - 1 } ) \left[(p^{-1}q)^{}X_2^{-}\left(z_1\right) X_2^{-}\left(z_2\right) X_1^{-}(w)-\left(q{+}p^{-1}\right) X_2^{-}\left(z_1\right) X_1^{-}(w) X_2^{-}\left(z_2\right)\right.\right. \\
			& \left.\left.\qquad\qquad\qquad\qquad+\,X_1^{-}(w) X_2^{-}\left(z_1\right) X_2^{-}\left(z_2\right)\right]\right\}+\left\{z_1 \leftrightarrow z_2\right\}=0 .
		\end{align}
Case $m=1$,
\begin{align}\notag
			& \left\{( z _ { 2 } p ^ { - 1 } - z _ { 1 } q ) \left[(p^{-1}q)X_1^{+}\left(z_1\right) X_1^{+}\left(z_2\right) X_2^{+}(w)-\left(q{+}p^{-1}\right) X_1^{+}\left(z_1\right) X_2^{+}(w) X_1^{+}\left(z_2\right)\right.\right. \\
			& \left.\left.\qquad\qquad\qquad\qquad+\,X_2^{+}(w) X_1^{+}\left(z_1\right) X_1^{+}\left(z_2\right)\right]\right\}+\left\{z_1 \leftrightarrow z_2\right\}=0, \\ \notag
			& \left\{X_2^{+}\left(z_1\right) X_2^{+}\left(z_2\right) X_1^{+}(w)-\left(q+p^{-1}\right) X_2^{+}\left(z_1\right) X_1^{+}(w) X_2^{+}\left(z_2\right)\right. \\
			& \left.\qquad+\,(p^{-1}q)X_1^{+}(w) X_2^{+}\left(z_1\right) X_2^{+}\left(z_2\right)\right\}+\left\{z_1 \leftrightarrow z_2\right\}=0,\\ \notag
			& \left\{( z _ { 2 } q - z _ { 1 } p ^ { - 1 }  ) \left[X_1^{-}\left(z_1\right) X_1^{-}\left(z_2\right) X_2^{-}(w)-\left(q+p^{-1}\right) X_1^{-}\left(z_1\right) X_2^{-}(w) X_1^{-}\left(z_2\right)\right.\right. \\
			& \left.\left.\qquad\qquad+\,(p^{-1}q)X_2^{+}(w) X_1^{-}\left(z_1\right) X_1^{-}\left(z_2\right)\right]\right\}+\left\{z_1 \leftrightarrow z_2\right\}=0,
		\end{align}
\begin{equation}
		\begin{aligned}
			& \left\{(p^{-1}q)X_2^{-}\left(z_1\right) X_2^{-}\left(z_2\right) X_1^{-}(w)-\left(q+p^{-1}\right) X_2^{-}\left(z_1\right) X_1^{-}(w) X_2^{-}\left(z_2\right)\right. \\
			& \left.\qquad\qquad+\,X_1^{-}(w) X_2^{-}\left(z_1\right) X_2^{-}\left(z_2\right)\right\}+\left\{z_1 \leftrightarrow z_2\right\}=0 .
		\end{aligned}
	\end{equation}
	
	Case $m=0$,
	\begin{equation}
		\begin{aligned}
			& \left\{(p^{-1}q)X_1^{ +}\left(z_1\right) X_1^{+}\left(z_2\right) X_2^{ +}(w)-\left(q+p^{-1}\right) X_1^{ +}\left(z_1\right) X_2^{ +}(w) X_1^{ +}\left(z_2\right)\right. \\
			& \left.\qquad\qquad+\,X_2^{ +}(w) X_1^{ +}\left(z_1\right) X_1^{ +}\left(z_2\right)\right\}+\left\{z_1 \leftrightarrow z_2\right\}=0,
		\end{aligned}
	\end{equation}
	\begin{equation}
		\begin{aligned}
			& \left\{X_1^{ -}\left(z_1\right) X_1^{-}\left(z_2\right) X_2^{ -}(w)-\left(q+p^{-1}\right) X_1^{ -}\left(z_1\right) X_2^{ -}(w) X_1^{ -}\left(z_2\right)\right. \\
			& \left.\qquad+\,(p^{-1}q)X_2^{ -}(w) X_1^{ -}\left(z_1\right) X_1^{ -}\left(z_2\right)\right\}+\left\{z_1 \leftrightarrow z_2\right\}=0,
		\end{aligned}
	\end{equation}
	\begin{equation}
		\begin{aligned}
			& \left\{X_2^{ +}\left(z_1\right) X_2^{ +}\left(z_2\right) X_1^{ +}(w)-\left(q+p^{-1}\right) X_2^{ +}\left(z_1\right) X_1^{ +}(w) X_2^{ +}\left(z_2\right)\right. \\
			&\qquad \left.+\,(p^{-1}q)X_1^{ +}(w) X_2^{ +}\left(z_1\right) X_2^{ +}\left(z_2\right)\right\}+\left\{z_1 \leftrightarrow z_2\right\}=0 .
		\end{aligned}
	\end{equation}
	\begin{equation}
		\begin{aligned}
			& \left\{(p^{-1}q)X_2^{ -}\left(z_1\right) X_2^{ -}\left(z_2\right) X_1^{ -}(w)-\left(q+p^{-1}\right) X_2^{ -}\left(z_1\right) X_1^{ -}(w) X_2^{ -}\left(z_2\right)\right. \\
			&\qquad\qquad\left.+\,X_1^{ -}(w) X_2^{ -}\left(z_1\right) X_2^{ -}\left(z_2\right)\right\}+\left\{z_1 \leftrightarrow z_2\right\}=0 .
		\end{aligned}
	\end{equation}
	As all verifications are similar for the Serre relations, we just use (\ref{ser1 m=3}) and (\ref{ser2 m=2}) as examples. It follows from
 (\ref{X1X1 rel1 type2}) and (\ref{X1X2 rel1}) that
	\begin{gather}
		X_1^{+}\left(z_1\right) X_2^{+}(w) X_1^{+}\left(z_2\right)=\frac{\left(z_2-w\right) q p^{-1}}{z_2 q-w p^{-1}} X_1^{+}\left(z_1\right) X_1^{+}\left(z_2\right) X_2^{+}(w),\\
		X_2^{+}(w) X_1^{+}\left(z_1\right) X_1^{+}\left(z_2\right)=\frac{\left(z_1-w\right) q p^{-1}}{z_1 q-w p^{-1}} \frac{\left(z_2-w\right) q p^{-1}}{z_2 q-w p^{-1}} X_1^{+}\left(z_1\right) X_1^{+}\left(z_2\right) X_2^{+}(w),\\
		X_1^{+}\left(z_2\right) X_1^{+}\left(z_1\right) X_2^{+}(w)=\frac{z_2 q^{-1}-z_1 p}{z_2 p-z_1 q^{-1}} X_1^{+}\left(z_1\right) X_1^{+}\left(z_2\right) X_2^{+}(w),\\
		X_1^{+}(z_2) X_2^{+}(w) X_1^{+}(z_1)=\frac{(z_1-w) q p^{-1}}{z_1 q-w p^{-1}} \frac{z_2 q^{-1}-z_1 p}{z_2 p-z_1 q^{-1}} X_1^{+}(z_1) X_1^{+}(z_2) X_2^{+}(w),\\
		X_2^{+}(w) X_1^{+}(z_2) X_1^{+}(z_1){=}\frac{(z_2{-}w) q p^{-1}}{z_2 q{-}w p^{-1}} \frac{(z_1{-}w) q p^{-1}}{z_1 q{-}w p^{-1}} \frac{z_2 q^{-1}{-}z_1 p}{z_2 p{-}z_1 q^{-1}} X_1^{+}(z_1) X_1^{+}(z_2) X_2^{+}(w).
	\end{gather}
	Note that the coefficient of $X_1^{+}\left(z_1\right) X_1^{+}\left(z_2\right) X_2^{+}(w)$:
	\begin{equation*}
		\begin{aligned}
			p^{-1} q-\left(q-p^{-1}\right) \frac{\left(z_2-w\right) q p^{-1}}{z_2 q-w p^{-1}}+\frac{\left(z_1-w\right) q p^{-1}}{z_1 q-w p^{-1}} \frac{\left(z_2-w\right) q p^{-1}}{z_2 q-w p^{-1}}+p^{-1} q \frac{z_2 q^{-1}-z_1 p}{z_2 p-z_1 q^{-1}}\\
			-\left(q-p^{-1}\right) \frac{\left(z_1-w\right) q p^{-1}}{z_1 q-w p^{-1}} \frac{z_2 q^{-1}-z_1 p}{z_2 p-z_1 q^{-1}}+\frac{\left(z_2-w\right) q p^{-1}}{z_2 q-w p^{-1}} \frac{\left(z_1-w\right) q p^{-1}}{z_1 q-w p^{-1}} \frac{z_2 q^{-1}-z_1 p}{z_2 p-z_1 q^{-1}},
		\end{aligned}
	\end{equation*}
	which is $0$. Therefore (\ref{ser1 m=3}) holds true. Next, we consider (\ref{ser2 m=2}). By (\ref{X1X1 rel1 type2}) and (\ref{X1X2 rel1}) we have that
	\begin{gather}
		X_2^{+}\left(z_1\right) X_1^{+}(w) X_2^{+}\left(z_2\right)=\frac{w q-z_2 p^{-1}}{(w-z_2)qp^{-1}} X_2^{+}\left(z_1\right) X_2^{+}\left(z_2\right) X_1^{+}(w),\\
		X_1^{+}(w) X_2^{+}\left(z_1\right) X_2^{+}\left(z_2\right)=\frac{w q-z_1 p^{-1}}{(w-z_1)qp^{-1}} \frac{w q-z_2 p^{-1}}{(w-z_2)qp^{-1}} X_2^{+}\left(z_1\right) X_2^{+}\left(z_2\right) X_1^{+}(w),\\
		X_2^{+}\left(z_2\right) X_2^{+}\left(z_1\right) X_1^{+}(w)=-X_2^{+}\left(z_1\right) X_2^{+}\left(z_2\right) X_1^{+}(w),\\
		X_2^{+}\left(z_2\right) X_1^{+}(w) X_2^{+}\left(z_1\right)=-\frac{w q-z_1 p^{-1}}{(w-z_1)qp^{-1}} X_2^{+}\left(z_1\right) X_2^{+}\left(z_2\right) X_1^{+}(w),\\
		X_1^{+}(w) X_2^{+}\left(z_2\right) X_2^{+}\left(z_1\right)=-\frac{w q-z_1 p^{-1}}{(w-z_1)qp^{-1}} \frac{w q-z_2 p^{-1}}{(w-z_2)qp^{-1}} X_2^{+}\left(z_1\right) X_2^{+}\left(z_2\right) X_1^{+}(w) .
	\end{gather}
	So the coefficient of the term $X_2^{+}\left(z_1\right) X_2^{+}\left(z_2\right) X_1^{+}(w)$ becomes
	\begin{equation*}
		\begin{aligned}
			\left(z_1 p^{-1}-z_2 q\right)\left\{1-\left(q{-}p^{-1}\right) \frac{w q{-}z_2 p^{-1}}{\left(w{-}z_2\right) q p^{-1}}+p^{-1} q \frac{w q{-}z_1 p^{-1}}{\left(w{-}z_1\right) q p^{-1}} \frac{w q{-}z_2 p^{-1}}{\left(w{-}z_2\right) q p^{-1}}\right\}\\
			+\left(z_2 p^{-1}-z_1 q\right)\left\{-1+\left(q-p^{-1}\right) \frac{w q{-}z_1 p^{-1}}{\left(w{-}z_1\right) q p^{-1}}-p^{-1} q \frac{w q{-}z_1 p^{-1}}{\left(w{-}z_1\right) q p^{-1}} \frac{w q{-}z_2 p^{-1}}{\left(w{-}z_2\right) q p^{-1}}\right\}
		\end{aligned}
	\end{equation*}
	equal to $0$, (\ref{ser2 m=2}) holds true. Similarly we can verify other Serre relations.
	
Now we move to the general $N$ case. Just like the $N=3$ case, we divide the case of $N+1$ it into two $N$ cases,
use induction to assume that all relations for the $N$ cases are known, it then suffices to check
the relations between $k_1^{ \pm}(z), e_1^{ \pm}(z), f_1^{ \pm}(z)$ and $k_N^{ \pm}(z)$, $e_{N-1}^{ \pm}(z), f_{N-1}^{ \pm}(z)$.  For this purpose, we write $L^{ \pm}(z)$ and $L^{ \pm}(w)^{-1}$ respectively as follows.
	\begin{gather}
		L^{ \pm}(z)=\left(\begin{array}{ccc}
			k_1^{ \pm}(z) & k_1^{ \pm}(z) f_1^{ \pm}(z) & \cdots \\
			e_1^{ \pm}(z) k_1^{ \pm}(z) & \cdots & \vdots \\
			\vdots & \cdots &
		\end{array}\right),
	\end{gather}
	and
	\begin{gather}
		L^{ \pm}(w)^{-1}=\left(\begin{array}{ccc}
			\cdots & \cdots & \vdots \\
			\vdots & \cdots & -f_{N-1}^{ \pm}(w) k_{N-1}^{ \pm}(w)^{-1} \\
			\cdots & -k_{N-1}^{ \pm}(w)^{-1} e_1^{ \pm}(w) & k_N^{ \pm}(w)^{-1}
		\end{array}\right).
	\end{gather}
	Using relations (\ref{RLL rel1 N=3}) and (\ref{RLL rel2 N=3}), we have
	\begin{gather}
		k_1^{ \pm}(z) k_N^{ \pm}(w)=k_N^{ \pm}(w) k_1^{ \pm}(z),\\
		\frac{z_{ \pm}-w_{\mp}}{z_{ \pm} q-w_{\mp} p^{-1}} k_N^{\mp}(w)^{-1} k_1^{ \pm}(z)=\frac{z_{\mp}-w_{ \pm}}{z_{\mp} q-w_{ \pm} p^{-1}} k_1^{ \pm}(z) k_N^{\mp}(w)^{-1},\\
k_1^{\epsilon}(z) e_{N-1}^{\epsilon'}(w)=e_{N-1}^{\epsilon'}(w) k_1^{\epsilon}(z),\\
k_1^{ \pm}(z) f_{N-1}^{\epsilon}(w)=f_{N-1}^{\epsilon}(w) k_1^{ \pm}(z),\\
k_N^{ \pm}(z) e_1^{\epsilon}(w)=e_1^{\epsilon}(w) k_N^{ \pm}(z),\\
k_N^{ \pm}(z) f_1^{\epsilon}(w)=f_1^{\epsilon}(w) k_N^{ \pm}(z),\\
f_1^{ \pm}(z) e_{N-1}^{\epsilon}(w)=e_{N-1}^{\epsilon}(w) f_1^{ \pm}(z),\\
e_1^{ \pm}(z) f_{N-1}^{\epsilon}(w)=f_{N-1}^{\epsilon}(w) e_1^{ \pm}(z),\\
f_1^{ \pm}(z) f_{N-1}^{\epsilon}(w)=f_{N-1}^{\epsilon}(w) f_1^{ \pm}(z),\\
e_1^{ \pm}(z) e_{N-1}^{\epsilon}(w)=e_{N-1}^{\epsilon}(w) e_1^{ \pm}(z),
	\end{gather}
	Note that all these relations are true for $R$-matrices in all situations. 
We have thus found all Drinfeld-type relations for $U_{p, q}(\widehat{\mathfrak{gl}}(m|n))$ and summarize them as follows.
	
	\begin{theorem}
		$U_{p, q}(\widehat{\mathfrak{gl}}(m|n))$ is an associative superalgebra with unit $1$ and Drinfeld current generators $X_i^{ \pm}(z), k_j^{ \pm}(z), i=1,2, \cdots, m+n-1, j=1,2, \cdots, m+n$, a central element $c$ and two nonzero complex parameters $p, q$ with invertible $k_i^{ \pm}(z)$.  The grading of the generators are: $\left[X_m^{ \pm}(z)\right]=1$ and zero otherwise. The defining relations are given by
		\begin{align}
			k_i^{ \pm}(z) k_j^{ \pm}(w)&=k_j^{ \pm}(w) k_i^{ \pm}(z),\\
			k_i^{+}(z) k_i^{-}(w)&=k_i^{-}(w) k_i^{+}(z), \quad i \leq m,\\
			\label{kiki}
			\frac{w_{-} p-z_{+} q^{-1}}{z_{+} p-w_{-} q^{-1}} k_i^{+}(z) k_i^{-}(w)&=\frac{w_{+} p-z_{-} q^{-1}}{z_{-} p-w_{+} q^{-1}} k_i^{-}(w) k_i^{+}(z), \quad m<i \leq m+n,
			\\
			\label{kikj}
			\frac{z_{ \pm}-w_{\mp}}{z_{ \pm} p-w_{\mp} q^{-1}} k_i^{\mp}(w)^{-1} k_j^{ \pm}(z)&=\frac{z_{\mp}-w_{ \pm}}{z_{\mp} p-w_{ \pm} q^{-1}} k_j^{ \pm}(z) k_i^{\mp}(w)^{-1}, \quad i>j,
		\end{align}
		\begin{align}\label{kjXi rel2}
			k_j^{ \pm}(z)^{-1} X_i^{\epsilon}(w) k_j^{ \pm}(z)&=X_i^{\epsilon}(w), \quad j-i \leq-1,
			\\\label{KjXi rel2}
			k_j^{ \pm}(z)^{-1} X_i^{\epsilon}(w) k_j^{ \pm}(z)&=X_i^{\epsilon}(w), \quad j-i \geq 2,\\
			k_i^{ \pm}(z)^{-1} X_i^{-}(w) k_i^{ \pm}(z)&=\frac{z_{\mp} p-w q^{-1}}{z_{\mp}-w} X_i^{-}(w), \quad i<m,\\
			k_i^{ \pm}(z)^{-1} X_i^{-}(w) k_i^{ \pm}(z)&=\frac{z_{\mp} q^{-1}-w p}{z_{\mp}-w} X_i^{-}(w), \quad m<i \leq m+n-1,\\
			k_{i+1}^{ \pm}(z)^{-1} X_i^{-}(w) k_{i+1}^{ \pm}(z)&=\frac{z_{\mp} q^{-1}-w p}{z_{\mp}-w} X_i^{-}(w), \quad i<m,\\
			k_{i+1}^{ \pm}(z)^{-1} X_i^{-}(w) k_{i+1}^{ \pm}(z)&=\frac{z_{\mp} p-w q^{-1}}{z_{\mp}-w} X_i^{-}(w), \quad m<i \leq m+n-1,
			\\
			k_i^{ \pm}(z) X_i^{+}(w) k_i^{ \pm}(z)^{-1}&=\frac{z_{ \pm} p-w q^{-1}}{z_{ \pm}-w} X_i^{+}(w), \quad i<m,\\
			\label{kiXi rel2}
			k_i^{ \pm}(z) X_i^{+}(w) k_i^{ \pm}(z)^{-1}&=\frac{z_{ \pm} q^{-1}-w p}{z_{ \pm}-w} X_i^{+}(w), \quad m<i \leq m+n-1,\\
			k_{i+1}^{ \pm}(z) X_i^{+}(w) k_{i+1}^{ \pm}(z)^{-1}&=\frac{z_{ \pm} q^{-1}-w p}{z_{ \pm}-w} X_i^{+}(w), \quad i<m,\\
			\label{ki+1Xi rel2}
			k_{i+1}^{ \pm}(z) X_i^{+}(w) k_{i+1}^{ \pm}(z)^{-1}&=\frac{z_{ \pm} p-w q^{-1}}{z_{ \pm}-w} X_i^{+}(w), \quad m<i \leq m+n-1,\\
			\label{kiXm rel1}
			k_i^{ \pm}(z)^{\epsilon} X_m^{\epsilon}(w) k_i^{ \pm}(z)^{-\epsilon}&=\frac{z_{\mp} p-w q^{-1}}{z_{\mp}-w} X_m^{\epsilon}(w), \quad i=m, m+1,\\
			\label{com rels1 Xi}
			\left(z q^{\mp 1}-w p^{ \pm 1}\right) X_i^{\mp}(z) X_i^{\mp}(w)&=\left(z p^{ \pm 1}-w q^{\mp 1}\right) X_i^{\mp}(w) X_i^{\mp}(z), \quad i<m,
			\\
			\left(w q^{\mp 1}-z p^{ \pm 1}\right) X_i^{\mp}(z) X_i^{\mp}(w)&=\left(w p^{ \pm 1}-z q^{\mp 1}\right) X_i^{\mp}(w) X_i^{\mp}(z), \quad m<i \leq m+n-1,\\
			\left\{X_m^{ \pm}(z), X_m^{ \pm}(w)\right\}&=0,\\
			(z-w)qp^{-1} X_i^{+}(z) X_{i+1}^{+}(w)&=\left(z q-w p^{-1}\right) X_{i+1}^{+}(w) X_i^{+}(z), \quad i<m,
		\end{align}
		\begin{align}
			(w-z)qp^{-1} X_i^{+}(z) X_{i+1}^{+}(w)&=\left(w q-z p^{-1}\right) X_{i+1}^{+}(w) X_i^{+}(z), \quad m \leq i \leq m+n-1,\\
			\left(z q-w p^{-1}\right) X_i^{-}(z) X_{i+1}^{-}(w)&=(z-w)qp^{-1} X_{i+1}^{-}(w) X_i^{-}(z), \quad i<m,\\
			\label{com rels7 Xi}
			\left(w q-z p^{-1}\right) X_i^{-}(z) X_{i+1}^{-}(w)&=(w-z)qp^{-1} X_{i+1}^{-}(w) X_i^{-}(z), \quad m \leq i \leq m+n-1,
			\\
			\left[X_i^{+}(z), X_j^{-}(w)\right]&= -\left(p-q^{-1}\right) \delta_{i j}\left(\delta\left(\frac{w}{z} q^c\right) k_{i+1}^{+}\left(w_{+}\right) k_i^{+}\left(w_{+}\right)^{-1}\right. \\
			& \quad\left.-\delta\left(\frac{w}{z} q^{-c}\right) k_{i+1}^{-}\left(z_{+}\right) k_i^{-}\left(z_{+}\right)^{-1}\right), \quad i, j \neq m, \notag
			\\
			\left\{X_m^{+}(z), X_m^{-}(w)\right\}&=  \left(p-q^{-1}\right)\left(\delta\left(\frac{w}{z} q^c\right) k_{m+1}^{+}\left(w_{+}\right) k_m^{+}\left(w_{+}\right)^{-1}\right. \\
			&\quad \left.-\delta\left(\frac{w}{z} q^{-c}\right) k_{m+1}^{-}\left(z_{+}\right) k_m^{-}\left(z_{+}\right)^{-1}\right),\notag
		\end{align}
		where $[X, Y] \equiv X Y-Y X$ stands for the commutator and $\{X, Y\} \equiv X Y+Y X$ the
		anti-commutator. The following are the Serre relations
		\begin{equation}
			\label{ser rel1 N}
			\begin{aligned}
				\bigl\{(p^{-1}q)&X_i^+\left(z_1\right)  X_i^{
					+}\left(z_2\right) X_{i+1}^{ +}(w)-\left(q{+}p^{-1}\right) X_i^{ +}\left(z_1\right) X_{i+1}^{ +}(w) X_i^{ +}\left(z_2\right) \\
				&	+\, X_{i+1}^{ +}(w) X_i^{ +}\left(z_1\right) X_i^{ +}\left(z_2\right)\bigr\}+\left\{z_1 \leftrightarrow z_2\right\}=0, \quad i \neq m,
			\end{aligned}
		\end{equation}
		\begin{equation}
			\begin{aligned}
				\bigl\{&X_i^{ -} (z_1)X_i^{
					-}(z_2) X_{i+1}^{ -}(w)-(q{+}p^{-1}) X_i^{ -}(z_1) X_{i+1}^{ -}(w) X_i^{ -}(z_2) \\
				&+\, (p^{-1}q)X_{i+1}^{ -}(w) X_i^{ -}(z_1) X_i^{ -}(z_2)\bigr\}+\left\{z_1 \leftrightarrow z_2\right\}=0, \quad i \neq m,
			\end{aligned}
		\end{equation}
		\begin{equation}
			\begin{aligned}
				& \left\{X_{i+1}^{ +}\left(z_1\right) X_{i+1}^{ +}\left(z_2\right) X_i^{ +}(w)-\left(q{+}p^{-1}\right) X_{i+1}^{ +}\left(z_1\right) X_i^{ +}(w) X_{i+1}^{ +}\left(z_2\right)\right. \\
				& \left.\quad+(p^{-1}q)X_i^{ +}(w) X_{i+1}^{ +}\left(z_1\right) X_{i+1}^{ +}\left(z_2\right)\right\}+\left\{z_1 \leftrightarrow z_2\right\}=0, \quad i \neq m-1,
			\end{aligned}
		\end{equation}
		\begin{equation}
			\begin{aligned}
				& \left\{(p^{-1}q)X_{i+1}^{-}\left(z_1\right) X_{i+1}^{-}\left(z_2\right) X_i^{ -}(w)-\left(q{+}p^{-1}\right) X_{i+1}^{ -}\left(z_1\right) X_i^{-}(w) X_{i+1}^{ -}\left(z_2\right)\right. \\
				& \left.\qquad\quad\,+X_i^{-}(w) X_{i+1}^{ -}\left(z_1\right) X_{i+1}^{ -}\left(z_2\right)\right\}+\left\{z_1 \leftrightarrow z_2\right\}=0, \quad i \neq m-1,
			\end{aligned}
		\end{equation}
		\begin{equation}
			\begin{aligned}
				&	\Bigr\{\left(z_1 p^{\mp 1}{-}z_2 q^{ \pm 1}\right)\Bigl[X_m^{ +}(z_1) X_m^{ +}(z_2) X_{m-1}^{ +}(w)-(q{+}p^{-1}) X_m^{ +}(z_1) X_{m-1}^{ +}(w) X_m^{ +}(z_2) \\
				&\qquad\qquad +\,(p^{-1}q)X_{m-1}^{ +}(w) X_m^{+}(z_1) X_m^{ +}(z_2)\Bigr]\Bigr\}+\{z_1 \leftrightarrow z_2\}=0,
			\end{aligned}
		\end{equation}
		\begin{equation}
			\begin{aligned}
				&\left\{\left(z_1 p^{\mp 1}{-}z_2 q^{ \pm 1}\right)\Bigl[(p^{-1}q)X_m^{-}(z_1) X_m^{ -}(z_2) X_{m-1}^{ -}(w)-(q{+}p^{-1}) X_m^{ -}(z_1) X_{m-1}^{ -}(w) X_m^{ -}(z_2)\right. \\
				&\qquad\qquad\qquad \qquad+\,X_{m-1}^{ -}(w) X_m^{ -}(z_1) X_m^{ -}(z_2)\Bigr]\Bigr\}+\{z_1 \leftrightarrow z_2\}=0,
			\end{aligned}
		\end{equation}
		\begin{equation}
			\begin{aligned}
				&\left\{\left(z_2 p^{\mp 1}{-}z_1 q^{ \pm 1}\right)\Bigl[(p^{-1}q)X_m^{ +}(z_1) X_m^{+}(z_2) X_{m+1}^{ +}(w)-(q{+}p^{-1}) X_m^{ +}(z_1) X_{m+1}^{ +}(w) X_m^{ +}(z_2)\right.\\
				& \qquad\qquad\qquad \qquad+\,X_{m+1}^{ +}(w) X_m^{ +}(z_1) X_m^{ +}(z_2)\Bigr]\Bigr\}+\{z_1 \leftrightarrow z_2\}=0,
			\end{aligned}
		\end{equation}
		\begin{equation}
			\label{ser rel8 N}
			\begin{aligned}
				&\left\{(z_2 p^{\mp 1}{-}z_1 q^{\pm 1})\Bigl[X_m^{ -}(z_1) X_m^{ -}(z_2) X_{m+1}^{ -}(w)-(q{+}p^{-1}) X_m^{ -}(z_1) X_{m+1}^{ -}(w) X_m^{ -}(z_2)\right. \\
				&\qquad\qquad\qquad+\,(p^{-1}q)X_{m+1}^{ -}(w) X_m^{ -}(z_1) X_m^{ -}(z_2)\Bigr]\Bigr\}+\left\{z_1 \leftrightarrow z_2\right\}=0.
			\end{aligned}
		\end{equation}
	\end{theorem}

	\section{HOPF SUPERALGEBRA STRUCTRUE OF $U_{p, q}(\widehat{\mathfrak{gl}}(m|n))$ }
	\begin{theorem}
		The algebra $U_{p, q}(\widehat{\mathfrak{gl}}(m|n))$  has a Hopf superalgebra structure given by the following formulae.
		
		\smallskip
		\textbf{Coproduct} $\Delta:$
		\begin{align*}
			\Delta\left(q^c\right)&=q^c \otimes q^c, \qquad  \Delta\left(p^c\right)=p^c \otimes p^c,\\
\Delta\left(k_j^{\pm}(z)\right)&=k_j^{\pm}\left(z q^{\frac{\pm c_2}{2}}\right) \otimes k_j^{\pm}\left(z q^{\mp\frac{c_1}{2}}\right),\quad j=1,\cdots,m+n,\\
			\Delta\left(X_i^{+}(z)\right)&=X_i^{+}(z) \otimes 1+\psi_i\left(z q^{\frac{c_1}{2}}\right) \otimes X_i^{+}\left(z q^{c_1}\right),\\
			\Delta\left(X_i^{-}(z)\right)&=1 \otimes X_i^{-}(z)+X_i^{-}\left(z q^{c_2}\right) \otimes \phi_i\left(z q^{\frac{c_2}{2}}\right), \quad i=1,2, \cdots, m+n-1,
		\end{align*}
		where $
		c_1=c \otimes 1, c_2=1 \otimes c$,
$\phi_i(z)=k_{i+1}^{+}(z) k_i^{+}(z)^{-1}$ and $\phi_i(z)=k_{i+1}^{+}(z) k_i^{+}(z)^{-1}$ are the following generating functions:
		$$
		\begin{aligned}
			& \psi_i(z)=\sum_{m \in \mathbb{Z}} \psi_i(m) z^{-m}=\left(\ell_{i i}^{-(0)}\right)^{-1} \ell_{i+1, i+1}^{-(0)} \exp \left(-\left(p-q^{-1}\right) \sum_{k>0} H_{-i k} q^{\frac{c}{2}} z^k\right), \\
			& \varphi_i(z)=\sum_{m \in \mathbb{Z}} \varphi_i(m) z^{-m}=\left(\ell_{i i}^{+(0)}\right)^{-1} \ell_{i+1, i+1}^{+(0)} \exp \left(\left(p-q^{-1}\right) \sum_{k>0} H_{i k} q^{-\frac{c}{2}} z^{-k}\right) .
		\end{aligned}
		$$

		\smallskip
		\textbf{Counit} $\epsilon:$
		\begin{align*}
			\epsilon\left(q^c\right)=1,\quad \epsilon\left(p^c\right)=1, \quad \epsilon\left(k_j^{ \pm}(z)\right)=1, \quad \epsilon\left(X_i^{ \pm}(z)\right)=0.
		\end{align*}
		
		\textbf{Antipode} $S:$
		\begin{align*}
			S\left(q^c\right)=q^{-c},\quad S\left(p^c\right)&=p^{-c}, \quad S\left(k_j^{ \pm}(z)\right)=k_j^{ \pm}(z)^{-1},\\
			S\left(X_i^{+}(z)\right)&=-\psi_i\left(z q^{-\frac{c}{2}}\right)^{-1} X_i^{+}\left(z q^{-c}\right),\\
			S\left(X_i^{-}(z)\right)&=-X_i^{-}\left(z q^{-c}\right) \phi_i\left(z q^{-\frac{c}{2}}\right)^{-1}.
		\end{align*}	
	\end{theorem}
	\begin{proof} Note that the tensor product is $\mathbb Z_2$-graded and the antipode is extended to the whole superalgebra, namely
		\begin{gather*}
			(a \otimes b)(c \otimes d)=(-1)^{[b][c]}(a c \otimes b d)
		\end{gather*}
		where $[a] \in \mathbb{Z}_2$ denotes the grading of the element $a$.
		
		We first prove  the comultiplication is a superalgebra homomorphism.
		
		For $m<i \leq m+n,$ thanks to relation (\ref{kiki}), we have
		$$
		\begin{aligned}
			& \Delta\left(\frac{w_{-} p-z_{+} q^{-1}}{z_{+} p-w_{-} q^{-1}} k_i^{+}(z) k_i^{-}(w)\right)=\frac{w q^{\frac{-c_1-c_2}{2}} p-z q^{\frac{c_1+c_2}{2}} q^{-1}}{z q^{\frac{c_1+c_2}{2}} p-w q^{\frac{-c_1-c_2}{2}} q^{-1}} \\
			& \quad \times\left(k_i^{+}\left(z q^{\frac{c_2}{2}}\right) \otimes k_i^{+}\left(z q^{-\frac{c_1}{2}}\right)\right)\left(k_i^{-}\left(w q^{-\frac{c_2}{2}}\right) \otimes k_i^{-}\left(w q^{\left.\frac{c_1}{2}\right)}\right)\right) \\
			& \quad=\frac{w q^{\frac{c_1+c_2}{2}} p-z q^{\frac{-c_1-c_2}{2}} q^{-1}}{z q^{\frac{-c_1-c_2}{2}} p-w q^{\frac{c_1+c_2}{2}} q^{-1}} \\
			& \quad \times\left(k_i^{-}\left(w q^{-\frac{c_2}{2}}\right) \otimes k_i^{-}\left(w q^{\frac{c_1}{2}}\right)\right)\left(k_i^{+}\left(z q^{\frac{c_2}{2}}\right) \otimes k_i^{+}\left(z q^{-\frac{c_1}{2}}\right)\right) \\
			& \quad=\Delta\left(\frac{w_{+} p-z_{-} q^{-1}}{z_{-} p-w_{+} q^{-1}} k_2^{-}(w) k_2^{+}(z)\right) .
		\end{aligned}
		$$
		Other relations regarding $k_i^{ \pm}(z)$ can be proved similarly.
		
		If $j-i \leq-2$, it is easy to check the following commutation relation between $\phi_j(z)$ and $X_i^{+}(w)$:
		\begin{gather*}
			\phi_j(z) X_i^{ -}(w) \phi_j(z)^{-1}=X_i^{ -}(w).
		\end{gather*}
		Then we compute that
		$$
		\begin{aligned}
			& \Delta\left(\phi_j(z) X_i^{-}(w) \phi_j(z)^{-1}\right)=\left(\phi_j\left(z q^{\frac{c_2}{2}}\right) \otimes \phi_j\left(z q^{-\frac{c_1}{2}}\right)\right) \\
			& \quad \times\left(1 \otimes X_i^{-}(w)+X_i^{-}\left(w q^{c_2}\right) \otimes \phi_i\left(w q^{\frac{c_2}{2}}\right)\right)\left(\phi_j^{-1}\left(z q^{\frac{c_2}{2}}\right) \otimes \phi_j^{-1}\left(z q^{-\frac{c_1}{2}}\right)\right) \\
			& \quad=1 \otimes X_i^{-}(w)+X_i^{-}\left(w q^{c_2}\right) \otimes \phi_i\left(w q^{\frac{c_2}{2}}\right) \\
			& \quad=\Delta\left(X_i^{-}(w)\right).
		\end{aligned}
		$$
		Similarly we can check other relations
between $\phi_j(z)$ and $X_i^{ \pm}(w)$, $\psi_i(z)$ and $X_i^{ \pm}(w)$.
		By (\ref{ki+1Xi rel2}) and (\ref{kiXm rel1}), we see that 
		\begin{gather*}
			\psi_m\left(z q^{\frac{c_1}{2}}\right) X_m^{ +}(w)=X_m^{ +}(w) \psi_m\left(z q^{\frac{c_1}{2}}\right),
		\end{gather*}
		which in turn implies that 
		$$
		\begin{aligned}
			\Delta\left(\left\{X_m^{+}(z), X_m^{+}(w)\right\}\right)= & \Delta\left(X_m^{+}(z)\right) \Delta\left(X_m^{+}(w)\right)+\Delta\left(X_m^{+}(w)\right) \Delta\left(X_m^{+}(z)\right) \\
			= & X_m^{+}(z) X_m^{+}(w) \otimes 1+X_m^{+}(z) \psi_m\left(w q^{\frac{c_1}{2}}\right) \otimes X_m^{+}\left(w q^{c_1}\right) \\
			& -\psi_m\left(z q^{\frac{c_1}{2}}\right) X_m^{+}(w) \otimes X_m^{+}\left(z q^{c_1}\right) \\
			& +X_m^{+}(w) X_m^{+}(z) \otimes 1+X_m^{+}(w) \psi_m\left(z q^{\frac{c_1}{2}}\right) \otimes X_m^{+}\left(z q^{c_1}\right) \\
			& -\psi_m\left(w q^{\frac{c_1}{2}}\right) X_m^{+}(z) \otimes X_m^{+}\left(w q^{c_1}\right) \\
			& +\psi_m\left(z q^{\frac{c_1}{2}}\right) \psi_m\left(w q^{\frac{c_1}{2}}\right) \otimes X_m^{+}\left(z q^{c_1}\right) X_m^{+}\left(w q^{c_1}\right) \\
			& +\psi_m\left(w q^{\frac{c_1}{2}}\right) \psi_m\left(z q^{\frac{c_1}{2}}\right) \otimes X_m^{+}\left(w q^{c_1}\right) X_m^{+}\left(z q^{c_1}\right) \\
			= & \left\{X_m^{+}(z), X_m^{+}(w)\right\} \otimes 1+\psi_m\left(z q^{\frac{c_1}{2}}\right) \psi_m\left(w q^{\frac{c_1}{2}}\right) \\
			& \otimes\left\{X_m^{+}\left(z q^{c_1}\right), X_m^{+}\left(w q^{c_1}\right)\right\}=0.
		\end{aligned}
		$$
		The coproduct relations between $X_i^{ \pm}(z)$ and $X_j^{ \pm}(w)$ can be proved similarly.
		The following relations can be derived from (\ref{kiXm rel1}): 
		\begin{gather*}
			\psi_m\left(z q^{\frac{c_1}{2}}\right) X_m^{-}\left(w q^{c_2}\right)=X_m^{-}\left(w q^{c_2}\right) \psi_m\left(z q^{\frac{c_1}{2}}\right),
			\\
			\phi_m\left(w q^{\frac{c_2}{2}}\right) X_m^{+}\left(z q^{c_1}\right)=X_m^{+}\left(z q^{c_1}\right) \phi_m\left(w q^{\frac{c_2}{2}}\right).
		\end{gather*}
		Then we verify that 
		$$
		\begin{aligned}
			\Delta\left(\left\{X_m^{+}(z),\right.\right. &\left.\left.X_m^{-}(w)\right\}\right)\\
			&= \left\{X_m^{+}(z), X_m^{-}\left(w q^{c_2}\right)\right\} \otimes \phi_m\left(w q^{c_2}\right)\\
			&\quad +\psi_m\left(z q^{\frac{c_1}{2}}\right) \otimes\left\{X_m^{+}\left(z q^{c_1}\right), X_m^{-}(w)\right\}  \\
			&=  \left(p{-}q^{-1}\right)\Bigl[\delta\left(\frac{w}{z} q^{c_1+c_2}\right) \phi_m\left(w q^{\frac{c_2}{2}+\frac{c_1+c_2}{2}}\right) \phi_m\left(w q^{-\frac{c_1}{2}+\frac{c_1+c_2}{2}}\right) \\
			& \quad -\delta\left(\frac{w}{z} q^{-c_1-c_2}\right) \psi_m\left(z q^{\frac{-c_2}{2}+\frac{c_1+c_2}{2}}\right) \psi_m\left(w q^{\frac{c_1}{2}+\frac{c_1+c_2}{2}}\right)\Bigr] \\
			&=  \left(p-q^{-1}\right) \Delta\left(\delta\left(\frac{w}{z} q^c\right) \phi_m\left(w_{+}\right)-\delta\left(\frac{w}{z} q^{-c}\right) \psi_m\left(z_{+}\right)\right) .
		\end{aligned}
		$$
		Finally, we check the Serre relations, take (\ref{ser rel1 N}) as an example.
		$$
		\begin{aligned}
			&\Delta\left(X_i^{+}\left(z_1\right) X_i^{+}\left(z_2\right) X_{i+1}^{+}(w)\right)\\
			&=X_i^{+}\left(z_1\right) X_i^{+}\left(z_2\right) X_{i+1}^{+}(w) \otimes 1+X_i^{+}\left(z_1\right) X_i^{+}\left(z_2\right) \psi_{i+1}\left(w q^{\frac{q}{2}}\right) \otimes X_{i+1}^{+}\left(w q^{c_1}\right)\\
			&\quad-X_i^{+}\left(z_1\right) \psi_i\left(z_2 q^{\frac{c_1}{2}}\right) X_{i+1}^{+}(w) \otimes X_i^{+}\left(z_2 q^{c_1}\right)\\
			&\quad+X_i^{+}\left(z_1\right) \psi_i\left(z_2 q^{\frac{c_1}{2}}\right) \psi_{i+1}\left(w q^{\frac{c_1}{2}}\right) \otimes X_i^{+}\left(z_2 q^{c_1}\right) X_{i+1}^{+}\left(w q^{c_1}\right)\\
			&\quad+\psi_i\left(z_1 q^{\frac{c_1}{2}}\right) X_i^{+}\left(z_2\right) X_{i+1}^{+}(w) \otimes X_i^{+}\left(z_1 q^{c_1}\right)\\
			&\quad-\psi_i\left(z_1 q^{\frac{c_1}{2}}\right) X_i^{+}\left(z_2\right) \psi_{i+1}\left(w q^{\frac{c_1}{2}}\right) \otimes X_i^{+}\left(z_1 q^{c_1}\right) X_{i+1}^{+}\left(w q^{c_1}\right)\\
			&\quad+\psi_i\left(z_1 q^{\frac{c_1}{2}}\right) \psi_i\left(z_2 q^{\frac{c_1}{2}}\right) X_{i+1}^{+}(w) \otimes X_i^{+}\left(z_1 q^{c_1}\right) X_i^{+}\left(z_2 q^{c_1}\right)\\
			&\quad+\psi_i\left(z_1 q^{\frac{c_1}{2}}\right) \psi_i\left(z_2 q^{\frac{c_1}{2}}\right) \psi_{i+1}\left(w q^{\frac{c_1}{2}}\right) \otimes X_i^{+}\left(z_1 q^{c_1}\right) X_i^{+}\left(z_2 q^{c_1}\right) X_{i+1}^{+}\left(w q^{c_1}\right),
		\end{aligned}
		$$
		and also we can get $\Delta\left(X_i^{+}\left(z_1\right) X_{i+1}^{+}(w) X_i^{+}\left(z_2\right)\right)$, $\Delta\left(X_{i+1}^{+}(w) X_i^{+}\left(z_1\right) X_i^{+}\left(z_2\right)\right)$ and corresponding equations by switching $z_1$ and $z_2$.
		Then
		$$
		\begin{aligned}
			\left(p^{-1} q\right) \Delta\left(X_i^{+}\left(z_1\right) X_i^{+}\left(z_2\right) X_{i+1}^{+}(w)\right)&-\left(q+p^{-1}\right) \Delta\left(X_i^{+}\left(z_1\right) X_{i+1}^{+}(w) X_i^{+}\left(z_2\right)\right) \\
			&+\Delta\left(X_{i+1}^{+}(w) X_i^{+}\left(z_1\right) X_i^{+}\left(z_2\right)\right)+\left\{z_1 \leftrightarrow z_2\right\}
		\end{aligned}
		$$
		$$
		=d(p, q) \Delta\left(X_i^{+}\left(z_1\right) X_i^{+}\left(z_2\right) X_{i+1}^{+}(w)\right),
		$$
		where $d(p, q)$ is
		$$
		\begin{aligned}
			& \left(p^{-1} q{-}(q-p^{-1}) \frac{(z_2{-}w) q p^{-1}}{z_2 q{-}w p^{-1}}+\frac{\left(z_1{-}w\right) q p^{-1}}{z_1 q{-}w p^{-1}} \frac{\left(z_2{-}w\right) q p^{-1}}{z_2 q{-}w p^{-1}}+p^{-1} q \frac{z_2 q^{-1}{-}z_1 p}{z_2 p{-}z_1 q^{-1}}\right) \\
			& +\left(-(q{-}p^{-1}) \frac{\left(z_1{-}w\right) q p^{-1}}{z_1 q{-}w p^{-1}} \frac{z_2 q^{-1}{-}z_1 p}{z_2 p{-}z_1 q^{-1}}+\frac{\left(z_2{-}w\right) q p^{-1}}{z_2 q{-}w p^{-1}} \frac{(z_1{-}w) q p^{-1}}{z_1 q{-}w p^{-1}} \frac{z_2 q^{-1}{-}z_1 p}{z_2 p{-}z_1 q^{-1}}\right).
		\end{aligned}
		$$
	which can be directly verified to be zero. 
Next we list the required relationships below.
		\begin{gather*}
			\left(z q^{\mp 1}-w p^{ \pm 1}\right) X_i^{\mp}(z) X_i^{\mp}(w)=\left(z p^{ \pm 1}-w q^{\mp 1}\right) X_i^{\mp}(w) X_i^{\mp}(z), \quad i<m,
		\end{gather*}
		for $i<m-1$,
		\begin{gather*}
			\psi_i\left(z_2 q^{\frac{c_1}{2}}\right) X_{i+1}^{+}(w)=\frac{z_2 p-w q^{-1}}{z_2-w} X_{i+1}^{+}(w) \psi_i\left(z_2 q^{\frac{c_1}{2}}\right) ,
		\end{gather*}
		for $i<m$,
		\begin{gather*}
			\psi_i\left(z_2 q^{\frac{c_1}{2}}\right) X_i^{+}\left(z_1\right)=\frac{z_2 q^{-1}-z_1 p}{z_2 p-z_1 q^{-1}} X_i^{+}\left(z_1\right) \psi_i\left(z_2 q^{\frac{c_1}{2}}\right).
		\end{gather*}
		
		We have therefore proved that the comultiplication is a superalgebra homomorphism. Next we prove
that the antipode is a superalgebraic homomorphism, which we use the $N=2$ case to show our argument. Note that the antipode obeys
$S(a b)=(-1)^{[a][b]} S(b) S(a)$, where $[a] \in \mathbb{Z}_2$ is the degree of $a$.
		
		For $m=1$ and $n=1$, We first prove $\phi\left(w q^{-\frac{c}{2}}\right) \psi\left(z q^{-\frac{c}{2}}\right)=\psi\left(z q^{-\frac{c}{2}}\right) \phi\left(w q^{-\frac{c}{2}}\right)$, which is straightforward.
		\begin{gather*}
			\phi\left(w q^{-\frac{c}{2}}\right) \psi\left(z q^{-\frac{c}{2}}\right)=k_2^{+}\left(w q^{-\frac{c}{2}}\right) k_1^{+}\left(w q^{-\frac{c}{2}}\right)^{-1} k_2^{-}\left(z q^{-\frac{c}{2}}\right) k_1^{-}\left(z q^{-\frac{c}{2}}\right)^{-1}=\\
			\frac{w q^{-c} p{-}z q^{-1}}{w q^{-c}{-}z} \frac{w{-}z q^{-c}}{w p{-}z q^{-c} q^{-1}} \frac{w p{-}z q^{-c} q^{-1}}{z q^{-c} p{-}w q^{-1}} \frac{z p{-}w q^{-c} q^{-1}}{w q^{-c} p{-}q^{-1} z} \frac{z q^{-c} p{-}w q^{-1}}{z q^{-c}{-}w} \frac{z{-}w q^{-c}}{z p{-}w q^{-c} q^{-1}}\\
			\times k_2^{-}\left(z q^{-\frac{c}{2}}\right) k_1^{-}\left(z q^{-\frac{c}{2}}\right)^{-1} k_2^{+}\left(w q^{-\frac{c}{2}}\right) k_1^{+}\left(w q^{-\frac{c}{2}}\right)^{-1}=\psi\left(z q^{-\frac{c}{2}}\right) \phi\left(w q^{-\frac{c}{2}}\right).
		\end{gather*}
		Therefore we have in this case. 
		$$
		S\left(\left\{X_1^{+}(z), X_1^{-}(w)\right\}\right)=-S\left(X_1^{-}(w)\right) S\left(X_1^{+}(z)\right)-S\left(X_1^{+}(z)\right) S\left(X_1^{-}(w)\right)
		$$
		$$
		\begin{aligned}
			= & -\psi_1\left(z q^{-\frac{c}{2}}\right)^{-1} \phi_1\left(w q^{-\frac{c}{2}}\right)^{-1}\left\{X_1^{+}\left(z q^{-c}\right), X_1^{-}\left(w q^{-c}\right)\right\} \\
			= & -\psi_1\left(z q^{-\frac{c}{2}}\right)^{-1} \phi_1\left(w q^{-\frac{c}{2}}\right)^{-1}\left(p-q^{-1}\right) \\
			& \times\left(\delta\left(\frac{w}{z} q^c\right) \phi_1\left(w^{-\frac{c}{2}}\right)-\delta\left(\frac{w}{z} q^{-c}\right) \psi_1\left(z q^{-\frac{c}{2}}\right)\right) \\
			= & \left(p-q^{-1}\right)\left(\delta\left(\frac{w}{z} q^{-c}\right) \phi_1\left(w^{-\frac{c}{2}}\right)^{-1}-\delta\left(\frac{w}{z} q^c\right) \psi_1\left(z q^{-\frac{c}{2}}\right)^{-1}\right) \\
			= & \left(p-q^{-1}\right) S\left(\delta\left(\frac{w}{z} q^c\right) \phi_1\left(w^{\frac{c}{2}}\right)-\delta\left(\frac{w}{z} q^{-c}\right) \psi_1\left(z q^{\frac{c}{2}}\right)\right) .
		\end{aligned}
		$$
		The other relations of the antipode
are shown exactly in the same manner. 
		Let
		$
		M: U_{p, q}(\widehat{\mathfrak{gl}}(m|n)) \otimes U_{p, q}(\widehat{\mathfrak{gl}}(m|n)) \rightarrow U_{p, q}(\widehat{\mathfrak{gl}}(m|n))
		$  be the multiplication given by tensor product.  We can easily check that
		$$
		\begin{aligned}
			& M(1 \otimes \epsilon) \Delta=i d=M(\epsilon \otimes 1) \Delta, \\
			& M(1 \otimes S) \Delta=\epsilon=M(S \otimes 1) \Delta .
		\end{aligned}
		$$
		Thus we have shown that the coproduct, the counit and the antipode give a Hopf superalgebra
		structure.
	\end{proof}
	\begin{remark}
		(1) The universal $R$-matrix of the two parameter quantum affine superalgebra of type $A$ can be realized as the Casimir element of certain Hopf pairing, with the help of explicit coproduct formula of all the Drinfeld loop generators, which will be considered in our next paper.
		
		(2) The $RLL$ realization of quantum affine superalgebras of $(BCD)_n^{(1)}$-types remain open. Recently, we have developed a uniform
 method for the two-parameter quantum affine (super) algebra $U_{r, s}\left[\operatorname{osp}(1|2)^{(1)}\right]$, $U_{r, s}\left[\operatorname{osp}(2|2)^{(2)}\right]$ and $U_{r, s}(A_2^{(2)})$ \cite{HZ} $($in fact, the one parameter case only provided the Drinfeld realization without proof$)$. Using the homomorphism in \cite{ZHJ} and the results of this paper,
 we believe that the quantum affine superalgebras of $(BCD)_n^{(1)}$-types (both one- and two-parameter cases) can be treated similarly.
	\end{remark}
	\section*{\textbf{ACKNOWLEDGMENTS}}
	The paper is supported by the NNSFC (Grant Nos.
	12171155, 12171303) and in part by the Science and Technology Commission of Shanghai Municipality (No. 22DZ2229014)
	and Simons Foundation TSM-00002518
		\section*{\textbf{AUTHOR DECLARATIONS}}
		\subsection*{\textbf{Conflict of Interest}}
		The authors have no conflicts to disclose.
			\subsection*{\textbf{Author Contributions}}
\textbf{Naihong Hu}: Supervision (equal) Validation (equal) Writing - review and editing (equal).
\textbf{Naihuan Jing}: Validation (equal) Writing - review and editing (equal).
\textbf{Xin Zhong}:Conceptualization (equal) Formal analysis (equal) Methodology (equal) Writing - original draft (equal).
	\section*{\textbf{DATA AVAILABILITY}}
	All data that support the findings of this study are included within the article.

\end{document}